\documentclass[11pt]{amsart}
\usepackage{amssymb,amsmath,amsfonts,amsthm}
\usepackage{graphicx}
\usepackage{color}
\usepackage{psfrag}
\usepackage[pagewise]{lineno}

\newtheorem{theorem}{Theorem}[section]
\newtheorem{lemma}[theorem]{Lemma}
\newtheorem{proposition}[theorem]{Proposition}
\newtheorem{corollary}[theorem]{Corollary}
\theoremstyle{definition}
\newtheorem{definition}[theorem]{Definition}

\newtheorem{remark}[theorem]{Remark}

\newtheorem{ltheorem}{Theorem} 

\def\real{\mathbb{R}}

\def\rational{\mathbb{Q}}
\def\integer{\mathbb{Z}}
\def\natural{\mathbb{N}}
\def\supp{\operatorname{supp}}
\def\d{\operatorname{dist}}
\def\dim{\operatorname{dim}}
\def\id{\operatorname{id}}
\def\cF{\mathcal{F}}
\def\cW{\mathcal{W}}

\def\rels{\sim^s}
\def\relu{\sim^u}

\def\cH{\mathcal{H}}

\def\cP{\mathcal{P}}
\def\cE{\mathcal{E}}
\def\cM{\mathcal{M}}
\def\cN{\mathcal{N}}
\def\nB{\mathbf{B}}
\def\nA{\mathbf{A}}

\def\quand{\quad\text{and}\quad}

\def\hA{\hat{A}}

\def\tmu{\tilde{\mu}}
\def\proj{\mathbb{R}P^{2d-1}}
\def\SL{SL(2,\real)}
\def\Sp{Sp(2d,\real)}
\def\G{\mathcal{G}}

\def\cH{\mathcal{H}}

\def\cP{\mathcal{P}}
\def\cE{\mathcal{E}}
\def\cM{\mathcal{M}}
\def\cN{\mathcal{N}}
\def\cB{\mathcal{B}}
\def\nB{\mathbf{B}}
\def\nA{\mathbf{A}}

\def\tB{\tilde{B}}

\def\tf{\tilde{f}}
\def\tM{\tilde{M}}
\def\tp{\tilde{p}}

\def\tz{\tilde{z}}
\def\tx{\tilde{x}}
\def\ty{\tilde{y}}
\def\tm{\tilde{m}}

\newcommand{\norm}[1]{{\left\lVert  #1  \right\rVert}}
\newcommand{\abs}[1]{{\left\lvert  #1  \right\rvert}}

\title[Positive Lyapunov exponents]
{Stably positive Lyapunov exponents for symplectic linear cocycles over partially hyperbolic diffeomorphisms}
\author{Mauricio Poletti}
\address{LAGA -- Universit\'e Paris 13, 99 Av. Jean-Baptiste Cl\'ement, 93430 Villetaneus, France.}
\email{mpoletti@impa.br}

\begin{document}

\begin{abstract}
We consider $\Sp$ cocycles over two classes of partially 
hyperbolic diffeomorphisms: having compact center leaves and  time one maps of Anosov flows. We prove that the Lyapunov exponents are non-zero in an open and dense set in the H\"older topology.
\end{abstract}

\maketitle

%\setcounter{tocdepth}{2}
%\tableofcontents

\section{Introduction}

Positiveness of Lyapunov exponents was widely studied in the past years,
some natural questions that are trying to be answered are:
\begin{itemize}
\item \emph{Positive Lyapunov exponent are common?} 
\item \emph{What is the generic behaviour, positive exponents or zero exponents?}
\item \emph{Are positive exponents stable?}
\end{itemize}

Avila \cite{Av11} proved for several topologies, including $C^r$ topology for $r\geq 0$, 
that there exists a dense, but not necessarily open or generic, set of cocycles taking values in $\SL$ with non zero Lyapunov exponents. This was generalized to $\Sp$ cocycles by Xu \cite{Xu15}. 

A result stated by Ma\~n\'e and proved by Bochi \cite{Boc02} expose that the generic behaviour in the $C^0$ topology is not positiveness of Lyapunov exponents unless you have a strong hyperbolicity:
$C^0$ generically $\SL$ cocycles are uniformly hyperbolic or have zero Lyapunov exponents.

In the case of more regular cocycles with some hyperbolic behaviour on the base map and bunching condition on the fibers, the generic behavior changes radically. Viana~\cite{Almost} proved that, when the base maps are non-uniformly hyperbolic and the cocycle take values in the group $SL(d,\real)$, the Lyapunov exponents are generically
non-zero. This was generalized by Bessa-Bochi-Cambrainha-Matheus-Varandas-Xu~\cite{BBCMVX} for any non-compact semi-simple Lie group (for example $\Sp$). 

These results were extended in the partially hyperbolic setting when the map is volume preserving, accessible and the cocycle takes values in $SL(d,\real)$ by Avila-Viana-Santamaria~\cite{ASV13}.

Here we deal with $\Sp$ cocycles over a different class of partially hyperbolic maps, not contained in the previous results: they are not volume preserving neither accessible and also are not non-uniformly hyperbolic (it has zero center Lyapunov exponent).

Specifically we deal with two types of maps, partially hyperbolic with compact center leaves and time one maps of Anosov flows that preserves measures with some product structure (for example SRB measures for an Anosov flow). We prove that linear cocycles over this class of maps have positive Lyapunov exponents in an open and dense set. 
We postpone the precise statements for section~\ref{s.statement}.

We also discuss about continuity of the Lyapunov exponents for these partially hyperbolic maps when the cocycle takes values in $\SL$. Continuity of the Lyapunov exponents was proved for hyperbolic base maps in \cite{BBB}.

An important technical tool in the proof of our main theorem is the closedness of $s$ and $u$-states for cocycles over partially hyperbolic maps.
In many works closedness of $s$ or $u$-states has been proved and used for specific cases (\cite{BBB}, \cite{Extremal}, \cite{ASV13}).
As we want to give a more general proof of this fact for general partially hyperbolic maps without any extra conditions on invariant measure of the base map and for smooth cocycles, that is an interesting result on his own, we postpone this proof for the appendix.

\subsection{Organization of the paper}
In sections~\ref{s.linear.cocycles} and \ref{s.PH} we recall some definitions of linear cocycles, Lyapunov exponents, partial hyperbolicity and we present the two classes of partially hyperbolic that we consider.

In section~\ref{s.statement} we give the precise statements of our results, in section~\ref{s.holonomies} and \ref{s.measure} we introduce some technical tools, and the concept of holonomy invariant disintegrations. 

We work separately with the disintegrations for partially hyperbolic maps with center leaves in section~\ref{s.disint.ccl} and for time one maps in \ref{s.disint.tom}.
In Section~\ref{s.technical} we recall some technical results already known that we prove for completeness.

In section~\ref{s.proof.teop}, section~\ref{s.proofAB} and \ref{s.proof.continuity} we conclude the proofs of the theorems.

In the appendix~\ref{appendix} we give a proof of the closedness of $s$ and $u$-states for general partially hyperbolic maps.
\medskip

\textbf{Acknowledgements.} Thanks to M. Viana for the orientation, E. Pujals, C. Matheus and Disheng Xu for the discussions and useful ideas. L. Backes, F. Lenarduzzi, K. Marin and A. Sanchez for the useful commentaries.

\section{Linear Cocycles}\label{s.linear.cocycles}
Let $\G$ be a linear subgroup of $GL(d,\real)$, for some $d\in\natural$, the \emph{linear cocycle} defined by a measurable matrix-valued function $A:M\rightarrow \G$ over 
an invertible measurable map $f:M\rightarrow M$ is the (invertible) map $F_A:M\times \real^d\rightarrow M\times \real^d$ 
given by
\begin{equation}
\nonumber F_A\left(x,v\right)=\left(f(x),A(x)v\right).
\end{equation}
Its iterates are given by $F^n_A\left(x,v\right)=\left(f^n(x),A^n(x)v\right)$ where
\begin{equation*}\label{def:cocycles}
A^n(x)=
\left\{
	\begin{array}{ll}
		A(f^{n-1}(x))\ldots A(f(x))A(x)  & \mbox{if } n>0 \\
		Id & \mbox{if } n=0 \\
		A(f^{n}(x))^{-1}\ldots A(f^{-1}(x))^{-1}& \mbox{if } n<0 \\
	\end{array}
\right.
\end{equation*}
Sometimes we denote this cocycle by $(f,A)$.

Let $\mu$ be an $f$-invariant probability measure on $M$ and suppose that $\log\norm{A}$ and $\log\norm{ A^{-1}}$
are integrable. By Kingman's sub-additive ergodic theorem, see \cite{FET}, 
the limits 
$$\begin{aligned}
   \lambda^+(A,\mu,x)=\lim_{n\to \infty}\frac{1}{n}\log \norm{A^n(x)}\quand \\ 
   \lambda^-(A,\mu,x)=\lim_{n\to \infty}-\frac{1}{n}\log \norm{(A^n(x))^{-1}}.
  \end{aligned}
$$
exist for $\mu$-almost every $x\in M$. When there is no risk of ambiguity we write just 
$\lambda^+(x)=\lambda^+(A,\mu,x)$.

By Oseledets~\cite{Ose68}, at $\mu$-almost every point
$x\in M$ there exist real numbers $\lambda^+(x)=\lambda_1 (x)>\cdots >\lambda_k(x)=\lambda^-(x)$
and a decomposition $\real^d=E^1_{x} \oplus \cdots \oplus E^k_{x}$ into vector subspaces such that
\begin{equation*}
A(x)E^i_{x}=E^i_{f(x)}
\text{ and } \lambda_i(x)=\lim_{\mid n\mid \to \infty}\frac{1}{n}\log\|A^n(x)v\|
\end{equation*}
for every non-zero $v\in E^i_{x}$ and $1\leq i \leq k$.

Let $\Sp$ be the symplectic subgroup of $GL(2d,\real)$, this means that there exists some non-degenerated skew-symmetric bilinear form $\omega:\real^{2d}\times\real^{2d}\to \real$ preserved by the group action, i.e: for every $v,w\in\real^{2d}$, $\omega(A v, A w)=\omega(v,w)$, in particular for dimension $2$, $\SL$ is a symplectic group for the area form $dx\wedge dy$. Observe that any $GL(2,\real)$ cocycle defines a $\SL$ cocycle by taking $\hA=\sqrt{(1/\det A)} A$, so positive exponent for $\hA$ means $\lambda^+(A)>\lambda_-(A)$.

Define $L(A,\mu)=\int \lambda^+ \,d\mu$, we say that $A$ \emph{has positive exponent} if
 $L(A,\mu)>0$. When $\mu$ is ergodic, as we are going to assume later, we have 
$L(A,\mu)=\lambda^+(x)$ for $\mu$-almost every $x\in M$.
 
\section{Partial hyperbolicity}\label{s.PH}
A diffeomorphism $f:M \to M$ of a compact $C^k$, $k>1$, manifold $M$ is said to be \emph{partially hyperbolic} if there
exists a non-trivial splitting of the tangent bundle
\begin{equation*}
TM=E^s\oplus E^c\oplus E^u
\end{equation*}
invariant under the derivative $Df$, a Riemannian metric $\|\cdot\|$ on $M$, and positive
continuous functions $\nu$, $\hat{\nu}$, $\gamma$, $\hat{\gamma}$ with $\nu$, $\hat{\nu}<1$
and $\nu<\gamma<{\hat\gamma}^{-1}<{\hat\nu}^{-1}$ such that, for any unit vector $v\in T_pM$,
\begin{alignat*}{2}
& \|Df(p)v \| < \nu(p) & \quad & \text{if } v\in E^s(p),
 \\
\gamma(p) < & \|Df(p)v \|  < {\hat{\gamma}(p)}^{-1} & & \text{if } v\in E^c(p),
 \\
{\hat{\nu}(p)}^{-1} < & \|Df(p)v\| & &  \text{if } v\in E^u(p).
\end{alignat*}
% (Equivalently, one could ask these conditions for some iterate; see Gourmelon~\cite{Go07}.)
All three sub-bundles $E^s$, $E^c$, $E^u$ are assumed to have positive dimension. From now on,
we take $M$ to be endowed with the distance $\d:M\times M\to \real$ associated to such a Riemannian structure.

Suppose that $f:M\to M$ is a partially hyperbolic diffeomorphism. The stable and unstable bundles $E^s$ and $E^u$
are uniquely integrable and their integral manifolds form two transverse continuous foliations
$\cW^s$ and $\cW^u$, whose leaves are immersed sub-manifolds of the same class of
differentiability as $f$. These foliations are referred to as the \emph{strong-stable} and
\emph{strong-unstable} foliations. They are invariant under $f$, in the sense that
$$
f(\cW^s(x))= \cW^s(f(x)) \quand f(\cW^u(x))= \cW^u(f(x)),
$$
where $\cW^s(x)$ and $\cW^s(x)$ denote the leaves of $\cW^s$ and $\cW^u$, respectively,
passing through any $x\in M$. 

A partially hyperbolic diffeomorphism $f:M\to M$ is called \emph{dynamically coherent} if there exist
invariant foliations $\cW^{cs}$ and $\cW^{cu}$ with smooth leaves tangent to $E^c \oplus E^s$ and
$E^c \oplus E^u$, respectively. Intersecting the leaves of $\cW^{cs}$ and $\cW^{cu}$ one obtains a center
foliation $\cW^c$ whose leaves are tangent to the center sub-bundle $E^c$ at every point.

\subsection{Class $\nA$: Compact center leaves}\label{ss.compact.leaves}
The first class to be consider is when the center leaves are compact manifolds.

Let $\tM=M/\cW^c$ be the quotient of $M$ by the center foliation and $\pi:M\to \tM$ be the quotient map. 
We say that the center leaves \emph{form a fiber bundle} if for any $\cW^c(x)\in \tM$ there is a neighbourhood 
$V \subset \tM$ of $\cW^c(x)$ and a homeomorphism 
$$h_x : V \times \cW^c(x)\to \pi^{-1}(V)$$
smooth along the verticals $\lbrace d\rbrace \times \cW^c(x)$ and mapping each vertical onto the corresponding center leaf $d$.

\begin{remark} If $f$ is a volume preserving, partially hyperbolic, dynamically coherent diffeomorphism in dimension 3 whose center foliation is absolutely continuous and whose generic center leaves 
are circles, then, according to Avila, Viana and Wilkinson \cite{AVW11}, all center
leaves are circles and they form a fiber bundle up to a finite cover. 
\end{remark}

We define the \emph{center Lyapunov exponents} of $f:M\to M$ as 
$$\lambda^c(f)^+(x)=\lim\frac{1}{n}\log \norm{Df^n\mid_{E^c(x)}}$$
and 
$$\lambda^c(f)^-(x)=\lim-\frac{1}{n}\log \norm{{Df^n\mid_{E^c(x)}}^{-1}}.$$

Again by the Oseledets theorem, this limits exist for $\mu$-almost every $x\in M$. 
We say that $f:M\to M$ has \emph{zero center Lyapunov exponent} if $\lambda^c(f)^+(x)=\lambda^c(f)^-(x)=0$ for $\mu$-almost every $x\in M$.

\begin{remark}
When all the center Lyapunov exponents of $f$ are non-zero, the problem falls in the hypothesis of Viana \cite{Almost}. In this work we deal with the case of zero center Lyapunov exponents, where the previous techniques fails. 
\end{remark} 

The fiber bundle condition gives that the quotient $\tM=M/\cW^c$ is a topological manifold, and the induced map 
$\tf:\tM\to \tM$ is a \emph{hyperbolic homeomorphism} in the following sense:
\begin{definition}
 Given any $\tx\in \tM$ and $\epsilon >0$, we define the
\emph{local stable} and \emph{unstable sets} of $\tx$ with respect to $\tf$ by
\begin{align*}
  W^s_{\epsilon} (\tx) &:= \left\{\ty\in \tM : \d_{\tM}(\tf^n(\tx),\tf^n(\ty))\leq\epsilon ,\ \forall
    n \geq 0\right\}, \\
  W^u_{\epsilon } (\tx) &:= \left\{\ty\in \tM : \d_{\tM}(\tf^n(\tx),\tf^n(\ty))\leq\epsilon ,\ \forall
    n \leq 0\right\},
\end{align*}
respectively.
We say that a homeomorphism $\tf:\tM \to \tM$ is \emph{hyperbolic}
  whenever there exist constants $C,\epsilon ,\tau>0$ and $\lambda\in
  (0,1)$ such that the following conditions are satisfied:

  \begin{itemize}
  \item $\; \d_{\tM}(\tf^n(y_1),\tf^n(y_2)) \leq C\lambda^n \d_{\tM}(y_1,y_2)$,
    $\forall \tx\in \tM$, $\forall y_1,y_2 \in W^s_{\epsilon } (x)$, $\forall
    n\geq 0$;

  \item $\; \d_{\tM}(\tf^{-n}(y_1), \tf^{-n}(y_2)) \leq C\lambda^n
    \d_{\tM}(y_1,y_2)$, $\forall \tx\in \tM$, $\forall y_1,y_2 \in W^u_{\epsilon } (\tx)$,
    $\forall n\geq 0$;

  \item If $\d_{\tM}(\tx,\ty)\leq\tau$, then $W^s_{\epsilon }(\tx)$ and
    $W^u_{\epsilon }(\ty)$ intersect in a unique point which is denoted by
    $[\tx,\ty]$ and depends continuously on $\tx$ and $\ty$. 
  \end{itemize}
\end{definition}
Let $\mu$ be an $f$-invariant ergodic measure, define $\tmu=\pi_{\ast}\mu$, we say that $\mu$ has \textit{projective product structure} if there exist measures $\mu^s$ 
on $W^s_{\epsilon}(\tx)$ and 
$\mu^u$ on $W^u_{\epsilon}(\tx)$ such that locally $\tmu\sim \mu^s\times \mu^u$ 
(see definition~\ref{def.prod.structure}), we also assume $\supp(\tmu)=\tM$.

From now on we refer to these partially hyperbolic maps as \emph{class $\nA$}.

\subsection{Class $\nB$: time one map of Anosov flows}\label{ss.time.one}
The second class of partially hyperbolic maps that we consider are the \emph{time one maps of Anosov flows}. 

We say that a flow $\phi_t:M\to M$ is an Anosov flow if there exists a splitting 
$$
TM=E^s\oplus X\oplus E^u,
$$
where $X(p)=\frac{d }{dt}\phi_t(p)\mid_{t=0}$, invariant under the derivative $D\phi_t$, where $E^s$ is contracted and $E^u$ is expanded by the derivative of $\phi_t$.

Let $f:M\to M$ be the time one map of a $C^2$ Anosov flow (i.e: $f=\phi_1$). This map is partially hyperbolic with center bundle $E^c=X$.

In this particular case the center bundle $E^c=X$ is integrable and $C^1$ (observe that $\cW^c(p)=\phi_\real(p)$).  
The bundles $E^c+E^u$ and $E^s+E^c$ are integrable and absolutely continuous (See \cite{3flows}). We call them center stable $\cW^{cs}$ and center unstable $\cW^{cu}$ foliations.

We say that a measure $\mu$ is an \emph{SRB measure}
for the flow if the disintegration of the measure along the center unstable leaves is in the Lebesgue class along these sub-manifolds (for a more detaills see \cite{3flows}). Observe that we are asking for the measure to be invariant for the flow not just for the time one map.

Recall that by the spectral decomposition theorem the non-wandering set of the flow can be decomposed into basic sets $\Omega(\phi_t)=\Delta_1\cup\cdots \cup \Delta_n$, where the $\Delta_i$ are invariant and the restriction $\phi_t$ is transitive in these sets, in particular the support of every ergodic measure is supported in one of these sets.

Also these basic sets coincide with the closure of the homoclinic class of its periodic points, as the support of an SRB measures is saturated by the center unstable foliation this implies that $\supp(\mu)=\Delta_i$ for some $i\in \{1,\dots,n\}$, moreover this set is an attractor.
Conversely, every non trivial attractor of an hyperbolic flow admits some  SRB measure. 

From now on we refer to these partially hyperbolic maps as \emph{class $\nB$}.
 \section{Statements}\label{s.statement}

Let us denote by $H^{\alpha}(M)$ the space of $\alpha$-H\"older continuous maps $A:M\to \Sp$ endowed with the $\alpha$-H\"older topology which is generated by norm
 \begin{displaymath}
 \norm{A}_{\alpha}:= \sup _{x\in M} \norm{A(x)} + \sup _{x\neq y} \dfrac{\norm{A(x)-A(y)}}{\d(x,y)^{\alpha}}.
\end{displaymath}

A cocycle generated by an $\alpha$-H\"older function $A:M\to \Sp$ is \textit{fiber-bunched} 
if there exists $C_3>0$ and $\theta <1$ such that
\begin{equation}\label{eq.FB}
\norm{A^n(x)}\norm{A^n(x)^{-1}}\min\lbrace \nu(x),\hat{\nu}(y)\rbrace ^{n \alpha}\leq C_3 \theta ^n %\; \textrm{for every} \; x\in M \; \textrm{and} \; n\geq 0
\end{equation}
for every $x\in M$ and $n\geq 0$.

Our first result, for partially hyperbolic maps of class $\nA$ is
\begin{ltheorem}\label{t.compact.leaves}
Let $f:M\to M$ be a partially hyperbolic with compact center leaves that form a fiber bundle and let 
$\mu$ be an $f$-invariant ergodic measure with zero center Lyapunov exponent and projective product structure. 
 Then the Lyapunov exponents relative to $\mu$ are non-zero in an open and dense subset of the fiber-bunched cocycles of
 $H^{\alpha}(M)$.
\end{ltheorem}

The second result is an analogous for class $\nB$
\begin{ltheorem}\label{t.time.one}
Let $f:M\to M$ be a $C^2$ time one map of an Anosov flow and $\mu$ an SRB invariant measure for the flow, then $L(A,\mu)>0$ in an open and dense subset of the fiber bunched cocycles of $H^\alpha(M)$.
\end{ltheorem}

These theorems will be consequence of:
\begin{ltheorem}\label{teoprincipal}
Let $f:M\to M$ be a partially hyperbolic map of class $\nA$ or $\nB$. 
Let $A:M \to \Sp$ be a fiber bunched cocycle
such that the restriction to some periodic compact center leaf $\cW^c(p)$, intersecting the support of $\mu$, of $f$ has positive Lyapunov exponents. 
Then $A$ is accumulated by open sets of cocycles with positive Lyapunov exponents.
\end{ltheorem}

In the theorem above, for the restriction of the cocycle to the periodic center leaf we take the natural invariant measure: 
\begin{itemize}
\item given by the disintegration in center leaves for class $\nA$,
\item the invariant measure equivalent to Lebesgue in the orbits of the flow for class $\nB$, 
\end{itemize}
this will be explained in more detail in section~\ref{s.measure}.

We also study sub-sets of continuity of the Lyapunov exponents for $\G=\SL$.
\begin{definition}
 Given an invertible measurable map $g:N\to N$ an invariant measure $\eta$ (not necessarily ergodic) and an integrable cocycle $A:N\to \SL$ 
 we say that $A$ is a \emph{continuity point for the Lyapunov exponents} 
 if for every $\{A_k\}_k$ converging to $A$ 
 we have that $\lambda^+_{A_k}:N\to \real$ converges in measure 
 to $\lambda^+_A:N\to \real$.
 
 Observe that as $A_k\to A$ this implies that $\sup_k \norm{A_k}$ is bounded and consequently 
 $\lambda^+_{A_k}$ is also bounded. Thus, since we are dealing with probability measures, 
 convergence in measure is equivalent to convergence in $L^1_\eta$.
\end{definition}

When the measure is ergodic this definition coincide with the classical definition of continuity, namely $L(A_k,\mu)\to L(A,\mu)$ whenever $A_k\to A$. An interesting question 
 is if in this context the Lyapunov exponents are continuous. 

In general this is not true. For example, take a volume preserving Anosov diffeomorphism $g:\mathbb{T}^2\to \mathbb{T}^2$, $\theta\in \real$ irrational and define
 $$ 
f:\mathbb{T}^3\to \mathbb{T}^3,\quad 
 f(x,y,t)=(g(x,y),t+\theta),
 $$
 this is a volume preserving partially hyperbolic diffeomorphism. By Wang-You \cite{WaY13} there exist discontinuity points for the Lyapunov exponents for monotonic cocycles in the smooth topology, let $A:\mathbb{T}^1\to \mathbb{T}^1$ be one of this with $t\mapsto t+\theta$. Then it is easy to see that the cocycle $\hA:\mathbb{T}^3\to\SL$, $\hA(x,y,t)=A(t)$ over $f$ is a discontinuity point for the Lyapunov exponents.
 
With an additional condition in a center leave we can find open sets of continuity.

\begin{ltheorem}\label{teo.continuity}
Let $f:M\to M$ be a partially hyperbolic of class $\nA$ or $\nB$. 
Let $A:M \to \SL$ be a fiber bunched cocycle whose restriction to some compact periodic center leaf, intersecting the support of $\mu$, has
positive Lyapunov exponent and is a continuity point for the Lyapunov exponents.
Then $A$ is accumulated by open sets restricted to which the Lyapunov exponents vary continuously.
\end{ltheorem}

An interesting question is whether these hypotheses are open or generic in some topology. 
 
Following Avila-Krikorian \cite{AvKr13}, given $\epsilon>0$, we say that a function $f:I\to \real$ is $\epsilon$-\emph{monotonic} 
if 
$$\frac{\abs{f(x)-f(y)}}{\abs{x-y}} >\epsilon.$$ 

This definition extends to functions defined on $S^1=\real/\integer$ and taking values on $S^1$ by considering the standard lift. 
We say that $A:S^1\to \SL$, is $\epsilon$-\emph{monotonic} if for every $w\in \real^2 \setminus \{0\}$ the function
$t\mapsto \frac{A(t)\cdot w}{\norm{A(t)\cdot w}}$ is $\epsilon$-monotonic.

Avila-Krikorian proved in \cite{AvKr13} that the Lyapunov exponents are continuous functions of $\epsilon$-\emph{monotonic} cocycles.
 Moreover, the set of $\epsilon$-monotonic cocycles is an open subset of $H^1(S^1)$.

Thus as corollary we have:
\begin{theorem}\label{teo.example.cont}
 Let $f:M\times S^1\to M\times S^1$ be a skew-product of the form $f(x,t)=(g(x),t+\theta(x))$, where 
 $g:M\to M$ is an Anosov diffeomorphism such that there exists a $g$-periodic point $p\in M$ with 
 $\theta(p)\in \real\setminus \rational$, and let $\mu$ be an $f$-invariant ergodic measure with projective product structure.
  Let $H^{\epsilon}_p\subset H^{1}(M\times S^1)$ be the open set of cocycles 
 $B:M\times S^1\to \SL$ such that $B(p,\cdot):S^1\to \SL$
 is $\epsilon$-monotonic. Then the Lyapunov exponents vary continuously in an open and dense subset of $H^{\epsilon}_p$.
\end{theorem}

We also have an equivalent result for time one maps:

\begin{theorem}\label{t.continuity.B}
 Let $f:M\to M$ be the time one map of an Anosov flow and $\mu$ an invariant SRB measure. Suppose there exists some periodic compact center circle, $S^1_p$, in the  in the support of $\mu$ with irrational rotation number.
  Let $H^{\epsilon}_p\subset H^{1}(M)$ be the open set of cocycles 
 $B:M\to \SL$ such that $B\mid_{S^1_p}:S^1\to \SL$
 is $\epsilon$-monotonic. Then the Lyapunov exponents vary continuously in an open and dense subset of $H^{\epsilon}_p$.
\end{theorem}

For volume preserving and accessible partially hyperbolic maps Liang, Marin and Yang~\cite{LMY}, proved that there exists an open and dense set of points of continuity for the Lyapunov exponents. 
%On the other hand Wang-You \cite{wang2013} find examples of discontinuity points for quasi periodic cocycles in every category $C^{\ell}$, $1\leq\ell\leq \infty$. They leave open the question if the set of points of continuity of Lyapunov exponents has an open and dense set.

\section{Invariant holonomies}\label{s.holonomies}
At this section we introduce the key notion that we are going to use in the proof of Theorem \ref{teoprincipal}, namely invariant holonomies.

Suppose that $f:M\to M$ is partially hyperbolic. 
 The stable and unstable foliations are, usually, \emph{not} transversely smooth:
the holonomy maps between any pair of cross-sections are not even Lipschitz continuous,
in general, although they are always $\gamma$-H\"older continuous for some $\gamma>0$.
Moreover, if $f$ is $C^2$ then these foliations are absolutely continuous, in the following
sense.

Let $d=\dim M$ and $\cF$ be a continuous foliation of $M$ with $k$-dimensional smooth
leaves, $0<k<d$. Let $\cF(p)$ be the leaf through a point $p\in M$ and $\cF(p,R)\subset\cF(p)$
be the neighbourhood of radius $R>0$ around $p$, relative to the distance defined by the
Riemannian metric restricted to $\cF(p)$. A \emph{foliation box} for $\cF$ at $p$ is the
image of an embedding
$$
\Phi:\cF(p,R) \times \real^{d-k} \to M
$$
such that $\Phi(\cdot,0)=\id$, every $\Phi(\cdot, y)$ is a diffeomorphism from $\cF(p,R)$
to some subset of a leaf of $\cF$ (we call the image a \emph{horizontal slice}),
and these diffeomorphisms vary continuously with $y \in \real^{d-k}$.
Foliation boxes exist at every $p\in M$, by definition of continuous foliation with smooth leaves.
A \emph{cross-section} to $\cF$ is a smooth co-dimension-$k$ disk inside a foliation box that
intersects each horizontal slice exactly once, transversely and with angle uniformly bounded from zero.

Then, for any pair of cross-sections $\Sigma$ and $\Sigma'$, there is a well defined \emph{holonomy map}
$\Sigma\to\Sigma'$, assigning to each $x\in\Sigma$ the unique point of intersection of $\Sigma'$
with the horizontal slice through $x$. The foliation is \emph{absolutely continuous} if all these
homeomorphisms map zero Lebesgue measure sets to zero Lebesgue measure sets.
That holds, in particular, for  the strong-stable and strong-unstable foliations of partially hyperbolic
$C^2$ diffeomorphisms and, in fact, the Jacobians of all holonomy maps are bounded by a uniform constant.

If $f:M\to M$ is of class $\nB$, his center stable and center unstable foliations are also absolutely continuous.
If we take two center manifolds $\cW^c(p)$ and $\cW^c(z)$ in the same strong stable manifold, from every point $t\in \cW^c(p)$ we can define the stable holonomy locally $h^s_{p,z}:I\subset \cW^c(p)\to \cW^c(z)$ and the Jacobian of $h^s_{p,z}$ vary continuously with the points $p$ and $z$. Actually the Jacobian is given by 
\begin{equation}\label{eq.jacobian.h}
Jh^s_{p,z}=\lim_{n\to\infty}\frac{Jf^n_c(t)}{Jf^n_c(h^s_{p,z}(t))},
\end{equation}
where $Jf^n_c(t)=\det Df^n\mid_{E^c_t}(t)$.
Analogously for the unstable holonomies.

When $f$ is of class $\nA$, we can extend the stable and unstable holonomies to be defined as a map from a entire center manifold to another:
\begin{definition}\label{d.honomy.ccl}
Given $\tx\in \tM$ and $\ty\in\tM$, such that $\ty\in W^s(\tx)$, we can define the \emph{stable holonomy}
$h^s_{\tx,\ty}:\cW^c(\tx)\to \cW^c(\ty)$ as the map which assigns to each $t\in \cW^c(\tx)$ $h^s_{\tx,\ty}(t)$ 
the first intersection between $\cW^s(t)$ and $\cW^c(\ty)$. 
Analogously, for $\tz\in W^u(\tx)$ we define the \emph{unstable holonomy} 
$$h^u_{\tx,\tz}:\cW^c(\tx)\to \cW^c(\tz)$$
changing stable by unstable manifolds.
\end{definition}

\subsection{linear holonomies}\label{ss.linear.holonomies}
We say that $A$ admits \textit{strong stable holonomies} if there exist, for every $p\textrm{ and }q \in M$ with $p\in \cW^{s}_{loc}(q)$,
 linear transformations $H^{s,A}_{p,q}:\real^{2d}\to \real^{2d}$ 
with the following properties:
\begin{enumerate}
\item $H^{s,A}_{f^j(p),f^j(q)}=A^j(q)\circ H^{s,A}_{p,q}\circ A^j(p)^{-1}$ for every $j\geq 1$,
\item $H^{s,A}_{p,p}=id$ and, for $w\in \cW^s(p)$, $H^{s,A}_{p,q}=H^{s,A}_{w,q}\circ H^{s,A}_{p,w}$,
\item there exists some $L>0$ (that does not depend on $p,q$) such that $\lVert H^{s,A}_{p,q}-id \rVert \leq L\d(p,q)^{\alpha}$.
\end{enumerate}
These linear transformations are called \textit{strong stable holonomies}.

Analogously if $p\in \cW^{u}_{loc}(q)$ we have the \textit{strong unstable holonomies}. When there is no ambiguity we write $H^{s}_{p,q}$ instead of $H^{s,A}_{p,q}$.
\begin{remark}
The fiber bunched condition \eqref{eq.FB} implies the existence of the strong stable and strong unstable holonomies (see \cite{ASV_ORIG}). 
\end{remark}

Let $\proj$ be the real projective space, i.e: the quotient space of $\real^{2d}\setminus\{0\}$ by the equivalence 
relation $v\sim u$ if there exist $c\in\real\setminus\{0\}$, such that, $v=c u$. Every invertible linear transformation $B$ induce a
projective transformation, that we also denote by $B$, 
$$B:\proj\to \proj \quand B\left[v\right]=\left[B(v)\right].$$ 
We also denote by $F:M\times \proj \to M\times \proj$ the induced projective cocycle.

\section{Measure disintegration}\label{s.measure}

Given a measurable partition $\cP$ of $M$, by Rokhlin disintegration theorem (see \cite{FET}) there exists a family of measures $\{\mu_P\}_{P\in \cP}$ such that for every measurable set $X\in M$ 
\begin{itemize}
\item $P\to \mu_P(X)$ is measurable,
\item $\mu_P(X)=1$ and 
\item $\mu(X)=\int_M \mu_P (X) d\tilde{\mu}(P)$.
\end{itemize}
Moreover, such disintegrantion is essentially unique \cite{Rok62}.

In general the partition by the invariant foliations is not measurable, to overcome this problem we disintegrate our measure locally as we explain now.

Denote by $s$, $u$ and $c$ the dimensions of $E^s$, $E^u$ and $E^c$ respectively. We call $C_x\in M$ a foliated box centred in $x\in M$ if there exists a continuous function 
$\Phi:[-1,1]^d\to C$ such that 
\begin{itemize}
\item $\Phi(0)=x$,
\item for every $y\in [-1,1]^{d-c}$, 
$$
\cW^c_{C_x}(\Phi(y,0)):=\Phi(y,[-1,1]^c)\subset \cW^c(\Phi(y,0)),
$$
\item for every $y\in [-1,1]^u$, 
$$\cW^{cs}_{C_x}(\Phi(0_s,y,0)):=\Phi([-1,1]^s\times \{y\} \times [-1,1]^c)\subset \cW^{cs}(\Phi(0_s,y,0_c)),$$
\item  for every $x\in [-1,1]^s$, 
$$
\cW^{cu}_{C_x}(\Phi(x,0_u,0)):=\Phi(\{x\}\times [-1,1]^u \times [-1,1]^c)\subset \cW^{cu}(\Phi(x,0_u,0_c)),
$$
\end{itemize}
where $0_i$ is the zero vector of $R^i$. 

Given a foliated box $C_x$ lets call $\Sigma_x=\Phi([-1,1]^{d-c}\times\{0_c\})$ and $\pi:C_x\to \Sigma_x$ the natural projection given by $\pi(\Phi(x_s,y_u,z_c))=\Phi(x_s,y_u,0_c)$, observe that $\Sigma_x$ has a product structure $\Sigma^s\times \Sigma^u$ given by $\Phi$.

\begin{remark}\label{r.center.disint}
When $f$ is of class $\nA$, we can actually take 
$\Phi:[-1,1]^{d-c}\times \cW^c(x)\to C$, because the center foliation form a fiber bundle.

Also, in this case the partition by center leaves is a measurable partition, so using Rokhlin disintegration we can define a family of measures $\tx\mapsto \mu^c_{\tx}$, such that $\mu^c_{\tx}$ is concentrated in $\cW^c(\tx)$ and 
$$
\mu=\int \mu^c_{\tx} d\tmu(\tx).
$$
Moreover, we are going to prove (see section~\ref{s.disint.ccl}) that $\tx\to \mu^c_{\tx}$ is continuous and defined everywhere.
\end{remark}

\begin{definition}\label{def.prod.structure}
We say that $\mu$ has \emph{projective product structure} if $\pi_*\mu\mid C_x$ is absolutely continuous with a product measure $\mu^s\times \mu^u$, where $\mu^*$ is a measure on $\Sigma^*$ for $*=s,u$.
\end{definition}

\begin{remark}
When $f$ is the time one map of an Anosov flow, $\mu$ is an SRB measure for the flow if and only if for every $x$ the disintegration of $\mu\mid C_x$ given by the partition $\{\cW^{cu}_{C_x}(y)\}_{y\in C_x}$ gives measures absolutely continuous with respect to the Lebesgue measure on $\cW^{cu}(y)$.

Observe that in this case the partition in center manifolds (orbits of the flow) is not measurable, but because the measure is invariant by the flow, the local disintegration in the center leaves in a foliated box is actually absolutely continuous with respect to Lebesgue.
\end{remark}

If we have some periodic compact center leaf $\cW^c(p)$ of period $k$, we have a natural $f^k$-invariant measure $\mu^c_p$ supported in this leaf, for the class $\nA$ is just $\mu^c_{\tp}$ and for the class $\nB$ is the absolutely continuous measure induced by the flow.
%
%Now we can state our theorem that will be the principal tool to prove theorem~\ref{t.compact.leaves} and \ref{t.time.one}.

\subsection{Invariance principle}
Let $P:M\times \proj\to M$ be the projection to the first coordinate, $P(x,v)=x$, and let $m$ be an $F$-invariant measure on 
$M\times \proj$ such that $P_{\ast}m=\mu$.
By Rokhlin \cite{Rok62} we can disintegrate the measure $m$ into $\lbrace m_x\mbox{, }x \in M\rbrace$ with
 respect to the partition 
 $$
 \mathcal{P}=\lbrace \{p\}\times \proj\mbox{, }p\in M\rbrace,
 $$

A measure $m$ that projects on $\mu$ is called:
\begin{itemize}
 \item[($a$)] $s$-\emph{invariant} if there exists a full $\mu$-measure subset $M^s\subset M$ such that $m_z=(H^s_{y,z})_{\ast}m_y$ for every $y,z\in M^s$ in the same strong-stable leaf;
 \item[($b$)] $u$-\emph{invariant} if there exists a full $\mu$-measure subset $M^u\subset M$ such that $m_z=(H^u_{y,z})_{\ast}m_y$ for every $y,z\in M^u$ in the same strong-unstable leaf.
\end{itemize}
If both $a$ and $b$ are satisfied we say that $m$ is $su$-\emph{invariant}.

We recall the following result whose proof can be found in \cite{ASV13}.
\begin{theorem}[Invariance Principle]\label{t.IP}
Let $f:M\to M$ be a $C^1$ partially hyperbolic diffeomorphism and $A:M\to \G$ a linear cocycle defined over $f$. If $\mu$ is an $f$-invariant measure, such that $\lambda^+(A,\mu,x)=\lambda^-(A,\mu,x)=0$ almost everywhere, then every $F_A$-invariant measure $m$ such that $P_* m=\mu$ is $su$-invariant.
\end{theorem}

We say that a disintegration of $m$ is \emph{$u/c$-invariant} if, for every $x$ and $y$ in the support of $\mu$ and in the same unstable manifold, $\mu^c_x$-almost every $t\in \cW^c(x)$ (Lebesgue almost every for class $\nB$), 
$$
{H^u_{t,h(t)}}_* m_t=m_{h(t)}\text{ where }h(t)=\cW^c(y)\cap \cW^u(t).
$$  
Analogously, we call a disintegration \emph{$s/c$-invariant} changing unstable by stable. We say that the disintegration is \emph{$su/c$-invariant} if the disintegration is both $u/c$ and $s/c$-invariant. This property is going to play a major role in our argument so we now prove that if $m$ is $su$-invariant, then $m$ admits a disintegration $su/c$-invariant.

We separate the argument for each class of partially hyperbolic map that we consider.

\section{Disintegration for class $\nA$}\label{s.disint.ccl}
In this section we deal with the disintegration of $m$ when $f$ is of class $\nA$, as described in section~\ref{ss.compact.leaves}. 

We have two disintegrations, one of $m$ with respect to $\cP$, as defined in the previous section, and one of $\mu$ with respect to $\tilde{\mathcal{P}}=\lbrace \cW^c(\tp)\mbox{, }\tp\in \tM \rbrace$, as explained in remark~\ref{r.center.disint}

 We introduce a third disintegration $\lbrace \tm_{\tx}\mbox{, }\tx \in \tM\rbrace$ with respect to the partition 
 $\hat{\mathcal{P}}=\lbrace \cW^c(p)\times \proj\mbox{, }p\in M\rbrace$.

It is easy to see that 
$$\tm_{\tx}=\int m_x d\mu^c_{\tx},$$
see \cite[exercise~5.2.1]{FTEp}.

Because $\pi_{\ast}\mu$ has local product structure and $f$ has zero center Lyapunov exponent, by the Invariance Principle for smooth cocycles(\cite[Proposition~4.8]{Extremal}), there exists a continuous disintegration
 $\{\mu^c_{\tx}\}$ which is $su$-invariant everywhere in the support of $\tmu$. This means that:
\begin{itemize}
 \item[(a')] $\mu^c_{\tz}=(h^s_{\ty,\tz})_{\ast}\mu^c_{\ty}$ for every  $\ty,\tz\in \tM$ in the same stable set,
 \item[(b')] $\mu^c_{\tz}=(h^u_{\ty,\tz})_{\ast}\mu^c_{\ty}$ for every  $\ty,\tz\in \tM$ in the same unstable set.
\end{itemize}

Denote by $$f_{\tx}:\cW^c(\tx) \to \cW^c(\tf(\tx))$$
the restriction of $f$ to $\cW^c(\tx)$. 
By the $f$-invariance of $\mu$ we have that  
$${f_{\tx}}_{\ast}\mu^c_{\tx}=\mu^c_{\tf(\tx)}$$
 for $\tmu$-almost every $\tx\in\tM$. From the continuity
of the disintegration it follows that, this relation is true for every $x\in\supp \tmu$. 
So, for every $\tf^k$-fixed point $\tp\in \tM$ 
 $$
 {f^k_{\tp}}_{\ast}\mu^c_{\tp}=\mu^c_{\tp}.
 $$

%A measure $m$ that projects on $\mu$ is called:
%\begin{itemize}
% \item[($a'$)] $u$-\emph{invariant} if there exists a full $\mu$-measure subset $M^s\subset M$ such that $m_z=(H^s_{y,z})_{\ast}m_y$ for every $y,z\in M^s$ in the same strong-stable leaf;
% \item[($b'$)] $s$-\emph{invariant} if there exists a full $\mu$-measure subset $M^u\subset M$ such that $m_z=(H^u_{y,z})_{\ast}m_y$ for every $y,z\in M^u$ in the same strong-unstable leaf.
%\end{itemize}
%If both $a'$ and $b'$ are satisfied we say that $m$ is $su$-\emph{invariant}.
 For $\ty\in W^u(\tx)$ define the holonomy $H^{u}_{\tx,\ty}:\cW^c(\tx)\times \proj \to \cW^c(\ty)\times \proj$ by
\begin{equation}\label{eq.uholonomy}
H^u_{\tx,\ty}(t,v)=\left(h^u_{\tx,\ty}(t), H^u_{(\tx,t),(\ty,h^u_{\tx,\ty}(t))}v\right). 
\end{equation}

\begin{proposition}\label{p.sudisintA}
If $m$ is $su$-invariant then $m$ admits a continuous disintegration $\lbrace\tm_{\tx},\tx\in\tM\rbrace$
in $\supp(\tmu)$ with respect to $\hat{\cP}$ that is $su/c$-invariant.
\end{proposition}
\begin{proof}
We need to prove that for every $\tx \in \supp(\tmu)$ and $\ty \in \supp(\tmu)$ in the same
 $\ast$-leaf, $\ast \in \lbrace s,u \rbrace$, 
$$({H^*_{(\tx,x),(\ty,h^*_{\tx,\ty}(x))}})_{\ast}m_x=m_{h^*_{\tx,\ty}(x)}\mbox{ for }\mu^c_{\tx}\mbox{-almost every }x\in \cW^c_{\tx}.$$

Let us prove for $*=u$, the case $*=s$ is analogous. 
Take $\tx \in \tM$ and $\ty \in \tM$ in the same unstable leaf and such that $\mu^c_{\tx}$-almost every $x\in \cW^c(\tx)$ belongs to $M^u$ and
 $\mu^c_{\ty}$-almost every $y\in \cW^c(\ty)$ belongs to $M^u$. Then
\begin{equation*}
 ({H^u_{\tx,\ty}})_{\ast} \tm_{\tx}(B)=\int m_x\left({H^u_{\tx,\ty}}^{-1}(B)\right) d\mu^c_{\tx}(x),
\end{equation*}
Considering $B_y=B\cap \{y\}\times \proj$ we have
\begin{equation}\label{eq.invariance}
\begin{aligned}
 ({H^u_{\tx,\ty}})_{\ast} \tm_{\tx}(B)  &=		\int m_x\left({H^u_{(\tx,x),(\ty,h^u_{\tx,\ty}(x))}}^{-1}(B_{h^u_{\tx,\ty}(x)}\right) d\mu^c_{\tx}(x)\\
			  &= 		\int ({H^u_{(\tx,x),(\ty,h^u_{\tx,\ty}(x))}})_{\ast}m_x\left(B_{h^u_{\tx,\ty}(x)}\right) d\mu^c_{\tx}(x)\\
			  &=		\int m_{h^u_{\tx,\ty}}\left(B_{h^u_{\tx,\ty}(x)}\right) d\mu^c_{\tx}(x)\\
			  &=		\int m_{y}\left(B_y\right) d\left({h^u_{\tx,\ty}}_{\ast}\mu^c_{\tx}\right)(y)\\
			  &=		\int m_{y}\left(B_y\right) d\left( \mu^c_{\ty} \right)(y)=\tm_{\ty}(B)\\		
\end{aligned}
\end{equation}
proving that the disintegration $\{\tm_{\tx}\}$ is $u$-invariant. Analogously we can find a total measure set such that
 $\tm_{\tx}$ is $s$-invariant. Using \cite[Proposition~4.8]{Extremal} we conclude that $m$ admits a
disintegration continuous in $\supp(\tmu)$ with respect to the partition $\hat{\mathcal{P}}$, 
which is $s$ and $u$ invariant. 
Now the continuity implies that \eqref{eq.invariance} is true for every $\tx\in\supp(\tmu)$. Thus
$$({H^u_{(\tx,x),(\ty,h^u_{\tx,\ty}(x))}})_{\ast}m_x=m_{h^u_{\tx,\ty}(x)}\mbox{ for }\mu^c_{\tx}\mbox{-almost every }x\in \cW^c_{\tx}.$$
for every $\tx \in \supp(\tmu)$ and $\ty \in \supp(\tmu)$ in the same unstable leaf (analogously for the stable holonomies) as claimed.
\end{proof}

\section{Disintegration for class $\nB$}\label{s.disint.tom}
This section is devoted to prove the existence of an $su/c$-invariant disintegration of $m$ for the time one maps of Anosov flows, we prove a more general version of this result.

We say that a measure $\mu$ has a \emph{good product structure} if it has projective product structure and also for every foliated box the disintegration in center manifolds $\{\cW^c_{C_x}(y)\}_{y\in C_x}$ is absolutely continuous with respect to the Lebesgue measure.

Every SRB measure for the flow has a good projective product structure, the absolute continuity in the center manifold is just because the measure is invariant by the flow and the projective product structure is because of the absolute continuity of the center stable foliation (see \cite[Proposition~3.4]{ViY13} for a proof  of this fact).

\begin{theorem}\label{t.sucinv}
Let $f:M\to M$ be a $C^2$ partially hyperbolic, dynamically coherent diffeomorphism and $\mu$ an invariant measure with a good product structure, then if 
$m$ is an $su$-invariant measure then there exists a disintegration that is $su/c$-invariant. 
\end{theorem}
\begin{proof}
For any topological space $N$, we denote by $\cM(N)$ the space of measures on $N$ with the weak$^*$ topology.

There exist $M^s$ and $M^u$ of total measure with $s$-invariance and $u$-invariance respectively. Take $M'=M^s\cap M^u$, $x\in \supp(\mu)$
and a foliated box $C_x\subset M$; via a local chart we can write $C=\Sigma\times D^c_R$ where $\Sigma$ is a transversal section to the center foliation and $D^c_R$ is a disc of radius $R$ 
in a center manifold, let $\pi:C_x\to \Sigma$ be the natural projection given by the center discs, observe that the center stable and center unstable manifolds gives a product structure of $\Sigma=\Sigma^s\times \Sigma^u$ and by hypothesis $\mu_{\Sigma}\sim\mu^s\times \mu^u$.

By the absolute continuity of the center foliation and the continuity of the Jacobians of the stable and unstable holonomies of $f$ we have that 
the disintegration $\Sigma\to \cM(D^c_R),\quad x\mapsto \mu^c_x$ is of the form 
$\mu^c_x=\rho \mu^c$ where $\mu^c$ is the Lebesgue measure in $D^c_R$ and $\rho$ is continuous.

We write $x\rels y$ for $x,y\in \Sigma$ in the same center stable manifold, and $x\relu y$ if they are in the same center unstable manifold.
Take $r>0$ smaller than $R$ such that for every $x\rels y$, $h^s_{x,y}:D^c_r\to D^c_R$ is well defined. 

Fix some $x^s\in \Sigma^s$ such that $\mu^u\left(\{x^s\}\times \Sigma^u\cap \pi(M')\right)=1$ and fix some $x^u\subset \Sigma^u$ such that $\mu^c_{(x^s,x^u)}((x^s,x^u)\times D^c_R\cap M')=1$, by the absolute continuity in the center direction this implies that Lebesgue almost every $t\in (x^s,x^u)\times D^c_R$ is in $M'$; for simplicity denote $x_0=(x^s,x^u)$.

Take $\epsilon>0$ such that 
$$
D^c_{\epsilon}\subset \bigcap_{z\relu y\rels x_0}h^u_{y,z} h^s_{x_0,y}(D^c_r).
$$

Lets call $m^c_x=\int_{D^c_r}m_t d\mu^c_x(t)$. Fix $x\rels y$, for $t\in D^c_r$ lets call $h(t)=h^s_{x,y}(t)$ and define $(\cH^s_{x,y}m^c_x)_{h(t)}={H^s_{(x,t),(y,h(t))}}_* m_t$. If $m_t$ is well defined for almost every $t\in D^c_r$
then by the absolute continuity of the holonomies $(\cH^s_{x,y}m^c_x)_{t'}$ is defined for almost every $t'\in h^s_{x,y}(D^c_r)$. Analogously we can define $(\cH^u_{x,y}m^c_x)_{t'}$

Now define a new disintegration of $m$ in the box $\Sigma\times D^c_\epsilon$ in the following way: 
Fix the restriction of the original disintegration defined for almost every $(x_0,t)\in \{x_0\}\times D^c_r$, (in particular $m^c_{x_0}$ is well defined) 
extend the disintegration to every $y\relu x_0$ and $t'\in  h^u_{x_0,y}(D^c_r)$ by $(\cH^u_{x_0,y}m^c_{x_0})_{t'}$. Lets call this disintegration $m^c_y$, and then extend it for every $z\rels y\relu x_0$ by
$(\cH^s_{y,z}m^c_y)_{t''}$ for almost every $t''\in h^s_{y,z} h^u_{x_0,y}(D^c_r)$. 
Now by definition of $\epsilon$ this disintegration is well defined in $\Sigma\times D^u_\epsilon$, and it coincides with the original disintegration almost everywhere, lets denote this
new disintegration by $t\mapsto m^s_t$. 

By construction this disintegration is $s/c$-invariant. We claim that this new disintegration, 
with respect to the partition $\{x\}\times D^c_{\epsilon}\times \proj$, defined by 
$$
\Sigma\to \cM(D^c_r\times \proj),\quad m^{cs}_x=\int_{D^c_r}m^s_t d\mu^c_x(t),$$
is continuous.
Assuming this claim, we can define analogously, reducing $\epsilon$ if necessary, a disintegration $x\mapsto m^u_x$ that is $u/c$-invariant on $\Sigma\times \proj$. 
We have that $m^{cs}_x=m^{cu}_x$ for Lebesgue almost every $x\in \Sigma$, then as both are continuous we have the equality for every $x\in \Sigma$.
Then by the uniqueness of the Rokhlin disintegration, for every $x\in \Sigma$, $m^s_t=m^u_t$ for $\mu^c$ almost every $t\in D^c_{\epsilon}$. As this boxes cover $\supp(\mu)$ we conclude the theorem.

We are left to prove the claim.

We will prove that $x\mapsto m^{cs}_x$ is uniformly continuous varying $y\rels x$ and $z\relu y$.
Fix some continuous $\varphi:D^c_\epsilon \times \proj\to \real$ and let $y_n\to x$, $y_n\rels x$. 
\begin{equation*}
 \int \varphi (t,v) d m^{cs}_{y_n}=\int_{D^c_\epsilon} \int_{\proj} \varphi (t,v)dm^s_{y_n,t}(v) d\mu^c_{y_n}(t)
 \end{equation*}
\begin{equation*}
\int \varphi (t,v) d m^{cs}_{y_n}  =\int_{D^c_\epsilon} \int_{\proj} \varphi (t,v)d {H_{n,t}}_*m^s_{x,h_n(t)}(v) d\mu^c_{y_n}(t),
\end{equation*}
where $h_n(t)=h^s_{y_n,x}(t)$, and $H_{n,t}=H^s_{(x,h_n(t)),(y_n,t)}$ so we can write the integral as
$$
\int_{D^c_\epsilon} \int_{\proj} \varphi (t,H_{n,t} v)d m^s_{x,h_n(t)}(v) \rho(y_n,t) d\mu^c(t),
$$
changing variables we have: 
$$
\int_{h_n(D^c_\epsilon)} \int_{\proj} \varphi (h_n^{-1}(t),H_{n,h_n^{-1}(t)} v)d m^s_{x,t}(v) J_{\mu^c}h_n^{-1}(t) \rho(y_n,h_n^{-1}(t)) d\mu^c(t),
$$
where $ J_{\mu^c}h_n^{-1}$ is the Jacobian of $h_n^{-1}$ with respect to the Lebesgue measure $\mu^c$. Hence using that $h_n^{-1}\to \id_{\cW^c(x)}$, $J_{\mu^c}h_n^{-1}\to 1$, 
$H_{n,t}\to \id_{\proj}$ uniformly and the continuity of $\rho$ and $\varphi$, we conclude that $\int \varphi (t,v) d m^{cs}_{y_n}\to \int \varphi (t,v) d m^{cs}_{y}$.

For $z_n\relu x$, such that $z_n\to x$, we take $z_n\rels y_n\relu x_0$ and $x\rels x'\relu x_0$ (see figure~\ref{f.proof}),
\begin{figure}[ht]
    \centering
    \includegraphics[scale=0.6]{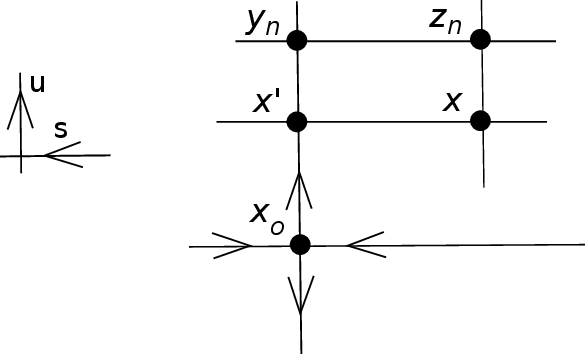}
    \caption{definition of $z_n$ and $x'$.}\label{f.proof}
\end{figure}
 by construction $y_n\to x'$, $y_n\relu x'\relu x_0$ and we have $u$-invariance on $W^u(x_0)$ so using the same arguments as before, changing stable by unstable 
$$
\int \varphi (t,v) d m^{cs}_{y_n}\to \int \varphi (t,v) d m^{cs}_{x'}.
$$
Now observe that
$$
\int \varphi (t,v) d m^{cs}_{z_n}  =\int_{D^c_\epsilon} \int_{\proj} \varphi (t,v)d {H^s_{n,t}}_*m^s_{y_n,h^s_n(t)}(v) d\mu^c_{z_n}(t),
$$
where $h^s_n=h^s_{y_n,z_n}$ and $H^s_{n,t}=H^s_{(y_n,h^s_n(t)),(z_n,t)}$ are such that $h^s_n\to h^s_{x',x}$ and $H^s_{n,t}\to H^s_{(x',h^s(t)),(x,t)}$.

Therefore, using the same calculations as before, it follows that
$\int \varphi (t,v) d m^{cs}_{z_n}\to \int \varphi (t,v) d m^{cs}_{x}$.

This proves the claim and concludes the proof.
\end{proof}

\section{Some technical results}\label{s.technical}

We can characterize the cocycles accumulated by cocycles with zero exponents, and also, in the $\SL$ case, the discontinuity points for the Lyapunov exponents using the next proposition:
\begin{proposition}\label{prop.discontinuity}
 If $A\in H^{\alpha}(M)$ is accumulated by cocycles with zero Lyapunov exponents, then there exists some $F_A$-invariant measure $m$, $su$-invariant, that projects
 to $\mu$. Also, if the cocycles takes values in $\SL$ and $A$ is a discontinuity point for the Lyapunov exponents, then every $F_A$-invariant measure $m$ that projects to $\mu$ is also $su$-invariant
\end{proposition}

\begin{proof}
Take $(A_k)_{k\in \natural}$ converging to $A$ such that $L(A_k,\mu)=0$.
Take $m_k$ to be $F_{A_k}$-invariant measures projecting to $\mu$, then by the invariance principle \cite{Extremal}, this measures are $su$-invariant. Take a subsequence such that $m_k$ converges to some $m$, by theorem~\ref{teo} this measure is $su$-invariant.

For the second part, lets suppose that an $\SL$ cocycle $A$ is a discontinuity point for the Lyapunov exponents.
We have that $L(A,\mu)>0$, otherwise $A$ is a continuity point, then by \cite[Proposition~3.1]{BP15} the conditional measures of  
 every $F$-invariant measure $m$ are of the form $m_x=a \delta_{E^{u}_x}+ b \delta_{E^{s}_x}$, 
  with $a+b=1$, so we can write $m=a m^u+ b m^s$ where $m^u$ and $m^s$ are measures projecting in $\mu$ with disintegration
  $m^u_x=\delta_{E^{u}_x}$ and $m^s_x=\delta_{E^{s}_x}$. The invariance of $E^u_x$ by unstable holonomies gives
  that $m^u$ is $u$-invariant, analogously $m^s$ is $s$-invariant.
  
  If $A$ is a discontinuity point then there exist $(A_k)_{k\in \natural}$ converging to $A$ such that $L(A_k,\mu)$ does 
  not converge to $L(A,\mu)$. Taking 
  $$\Psi_{A_k}:M\times \real P^1\to\real\quand \Psi_{A_k}(x,v)=\frac{\norm{A_k(x)v}}{\norm{v}}
  $$
  and $m^u_k$ the $u$-invariant measure projecting in $\mu$ as before (but for $A_k$) if $L(A_k,\mu)>0$ or 
  any $u$-invariant measure if $L(A_k,\mu)=0$, we have that 
  $$\int\Psi_{A_k}dm^u_k =L(A_k,\mu)$$ 
  does not converge to $$ L(A,\mu)=\int\Psi_{A}dm^u,$$
  taking a subsequence we can suppose that $m^u_k$ converges to some $m=a m^u+b m^s$ with $b\neq 0$, then
  we have that $m$ is $u$-invariant, so 
  $m^s=b^{-1}(m-a m^u)$ is also $u$-invariant, hence every $a' m^u+b' m^s$ is $u$-invariant.
  
  As $\lambda^+(A)=-\lambda^-(A)$ the proof of $s$-invariance follows analogously.
  \end{proof}

\subsection{Symplectic transvections}\label{ss:transvections} 
The results of this section can be found in the thesis of Cambrainha \cite{cambrainha}, for completeness we rewrite the statements and the proofs.

A linear map $\tau:\real^{2d}\to\real^{2d}$ is called a \emph{transvection} if there are a hyperplane $H\subset\real^{2d}$ and a vector $v\in H$ such that the restriction $\tau|_H$ is the identity on $H$ and for any vector $u\in\real^{2d}$, $\tau(u)-u$ is a multiple of $v$, say $\tau(u) - u=\lambda(u)v$ where $\lambda$ is a linear functional of $\real^m$ such that $H\subset\ker \lambda$.

For later use, let us recall the following characterization of symplectic transvections of $(\mathbb{R}^{2d}, \omega)$. If $\tau\in\Sp$ is a transvection of the form $\tau(u)=u+\lambda(u)v$, then
\begin{eqnarray*}
	 \omega(u,u')&=&\omega(\tau(u),\tau(u'))\\
	 &=&\omega(u+\lambda(u).v,u'+\lambda(u').v)\\
	 &=&\omega(u,u')-\lambda(u)\omega(u',v)+\lambda(u')\omega(u,v).\\
\end{eqnarray*}

Hence, $\lambda(u)\omega(u',v)=\lambda(u')\omega(u,v)$ for every $u,u'\in \real^{2d}$. Taking $u'$ such that $\omega(u',v)=1$ and denoting $a=\lambda(u')$, we get the general formula for a symplectic transvection:
$$\tau(v)=\tau_{v,a}(u)=u+a\cdot\omega(u,v)\cdot v$$

Of course, every transformation of this form is a symplectic transvection (fixing the hyperplane $(\mathbb{R}v)^\perp $).

\begin{lemma}\label{lemma:symp_transv} Let $(E,\omega)$ be a $2d$-dimensional symplectic vector space and let $V$ and $W$ be subspaces of $E$ with complementary dimensions (i.e., $\textrm{dim}(V)+\textrm{dim}(W)=2d$). Suppose that $V\cap W$ has dimension $k>0$. Then, there exist $k$ symplectic transvections $\sigma_1, \dots,\sigma_k$ arbitrarily close to the identity such that
$$(\sigma_1 \dots \sigma_k)(V)\cap W =\{0\}$$
\end{lemma}

\begin{proof} We proceed by induction on $k=\dim(V\cap W)$. Let $m=\dim(V)$, so that $k\leq \min\{m,2d-m\}$.

If $k=1$, then we can choose basis of $V$ and $W$ such that
$$V=\textrm{span}\{v_1,v_2,...,v_m\} \quad \textrm{and} \quad
W=\textrm{span}\{v_1,w_2,...,w_{2d-m}\}
$$
Note that $V_0=V+W$ and $V_1=(\mathbb{R} v_1)^\perp$ are hyperplanes of $E$.
Take $u_0 \notin V_0 \cup V_1$ and $\varepsilon>0$. Denote by $\sigma:E\to E$ the symplectic transvection associated to $u_0$ and $\varepsilon$:
$$
\sigma(u)= u+ \varepsilon\, \omega(u,u_0)\, u_0 \quad \text{ for every } u\in E.
$$
By definition, the symplectic transvection $\sigma$ becomes arbitrarily close to the identity as close to the identity by taking $\varepsilon>0$. Moreover, we claim that $\sigma(V)\cap W=\{0\}$. Indeed, take any $v\in \sigma(V)\cap W$, say
\begin{equation}\label{intersection}
\sum\limits_{i = 1}^m \alpha _i \, \sigma (v_i ) = v =  \beta _1 v_1  +
\sum\limits_{i = 2}^{2d - m} \beta _i \, w_i
\end{equation}
for some numbers $\alpha_1, \dots, \alpha_m$ and $\beta_1, \dots, \beta_{2d-m}$.
If we write $\sigma(v_i)=v_i+K_i u_0$ with $K_i=\varepsilon\, \omega(v_i, u_0)$, this equality implies that
\begin{equation}\label{lin ind}
(\alpha _1  - \beta _1 )v_1  + \sum\limits_{i = 2}^m {\alpha _i v_i }  - \sum\limits_{i = 2}^{2d - m} {\beta _i w_i }  + \left( {\sum\limits_{i = 1}^m {\alpha _i K_i } } \right)u_0  = 0.
\end{equation}
Since $u_0\notin V_0$, the set $\{v_1,...,v_m,w_2,...,w_{2d-m},u_0\}$ is a basis of $E$. Thus, all coefficients in \eqref{lin ind} are zero:
$$
\left\{\begin{array}{l}
\alpha _1  = \beta _1  \\
\alpha _i  = 0,\,\,\,\,\,i = 2,...,m \\
\beta _i  = 0,\,\,\,\,\,i = 2,...,2d - m \\
\sum\limits_{i = 1}^m \alpha _i K_i  = 0\\
\end{array}\right.
$$
From the second and the fourth equations above, we deduce that $\alpha _1 K_1  = 0$, i.e., $\varepsilon\, \alpha_1\,\omega(v_1, u_0)=0$. Since $u_0\notin V_1$ means that $\omega(v_1,u_0)\neq 0$, it follows that  $\alpha_1=0$ and, \emph{a fortiori}, $v=0$. This proves our claim, so that the first step of the induction argument is complete.

Let us now perform the general step of the induction. Suppose now that the lemma is true for $k-1$, and let $V\cap W=\textrm{span}\{v_1, \dots, v_k\}$. In this case, we can find basis for $V$ and $W$ such that
$$
\begin{aligned}
V&=\textrm{span}\{v_1, v_2, \dots, v_k, v_{k+1}, \dots, v_m\} \quad \textrm{and} \\
 W&=\textrm{span}\{v_1, \dots, v_k, w_{k+1}, \dots, w_{2d-m}\}.
\end{aligned}
$$
Observe that $V_0=V+W$ and $V_1=(\textrm{span}\{v_1,...,v_k\})^\perp$ are co-dimension $k$ subspaces of $E$. Fix $u_0\notin V_0\cup V_1$, $\varepsilon>0$ and consider the symplectic transvection
$$
\sigma_k(u):= u+ \varepsilon\,\omega(u, u_0)\,u_0, .
$$
Any element $v\in \sigma_k(V)\cap W$ can be written in a similar way to \eqref{intersection}. From this, we have that:
\begin{equation}
\sum\limits_{i = 1}^{k} {(\alpha _i  - \beta _i )v_i}  + \sum\limits_{i = k+1}^m {\alpha _i v_i }  - \sum\limits_{i = k+1}^{2d - m} {\beta _i w_i }  + \left( {\sum\limits_{i = 1}^m {\alpha _i K_i } } \right)u_0  = 0
\end{equation}
Since our choice $u_0\notin V_0$ implies that the vectors $\{v_1,...,v_m,w_{k+1},...,w_{2d-m},u_0\}$ are linearly independent, one has:
$$
\left\{\begin{array}{l}
\alpha _i  = \beta _i,\,\,\,\,\,i = 1, \dots, k \\
\alpha _i  = 0,\,\,\,\,\,i = k+1, \dots, m \\
\beta _i  = 0,\,\,\,\,\,i = k+1, \dots, 2d - m \\
\sum\limits_{i = 1}^m \alpha _i K_i  = 0\\
\end{array}\right.
$$
The second and fourth equations imply that $\sum\limits_{i = 1}^k \alpha _i K_i  = 0$, i.e.,
$\omega\left(\sum\limits_{i = 1}^k \alpha _i v_i,u_0\right)  = 0$. Therefore:
$$
(\sigma_k A)(V)\cap W=\left\{v=\sum\limits_{i=1}^k{\alpha_i\,v_i} \left|\ \ \omega\left(\sum\limits_{i = 1}^k \alpha _i v_i,u_0\right)  = 0\right.\right\}
$$
Because $u_0\notin V_1$, we deduce that $(\sigma_k A)(V)\cap W$ is a $(k-1)$-dimensional subspace. By  induction hypothesis, there are $\sigma_1, \dots,\sigma_{k-1}$ symplectic transvections arbitrarily close to the identity such that
$$
(\sigma_1\dots\sigma_{k-1}\sigma_k A)(V)\cap W = \{0\}
$$
This completes the proof of the lemma.
\end{proof}

The conclusion of the previous lemma is an open condition: $B(V)\cap W=\{0\}$ for every $B$ sufficiently close to the symplectic automorphism $\sigma$ provided by this lemma. By recursively using this fact and the previous lemma, we deduce that:
\begin{corollary}\label{cor:transv_finite}
Let $(E,\omega)$ be a $2d$-dimensional symplectic vector space and let $\{(V_j, W_j): 1\leq j\leq m\}$ be a finite collection of pairs subspaces of $E$ with complementary dimensions (i.e., $\textrm{dim}(V_j) + \textrm{dim}(W_j) = 2d$ for all $1\leq j\leq m$). Then, there exists a symplectic automorphism $\sigma$ arbitrarily close to the identity such that
$$\sigma(V_j)\cap W_j =\{0\} \ \ \forall\, j=1, \dots, m$$
\end{corollary}

\section{Proof of Theorem \ref{teoprincipal}}\label{s.proof.teop}

Observe that for class $\nA$, $\pi^{-1}$ of the $\tf$-periodic points are $f$-periodic center leaves, so the hyperbolicity of $\tf$ implies the existence of periodic compact center leaves.
 For class $\nB$ it is well know that every non-trivial basic set has a dense set of periodic orbits for the flow, then this are fixed compact center leaves (actually circles) for the time one map.
 
To simplify notation let us assume that there exists a $\tf$-fixed point $\tp$, i.e: $\tf(\tp)=\tp$ (since all the arguments 
and results are not affected by taking an iterate).

Fix this compact center leaf $K=\cW^c(p)$, so $f:K\to K$ has an invariant measure $\mu^c_K$.

Now we are going to define $h:K\to K$ that is a composition of stable and unstable holonomies.

\subsection{Class $\nA$}
First we define this map for $f$ of class $\nA$.
\begin{figure}[ht]
    \centering
    \includegraphics[scale=0.6]{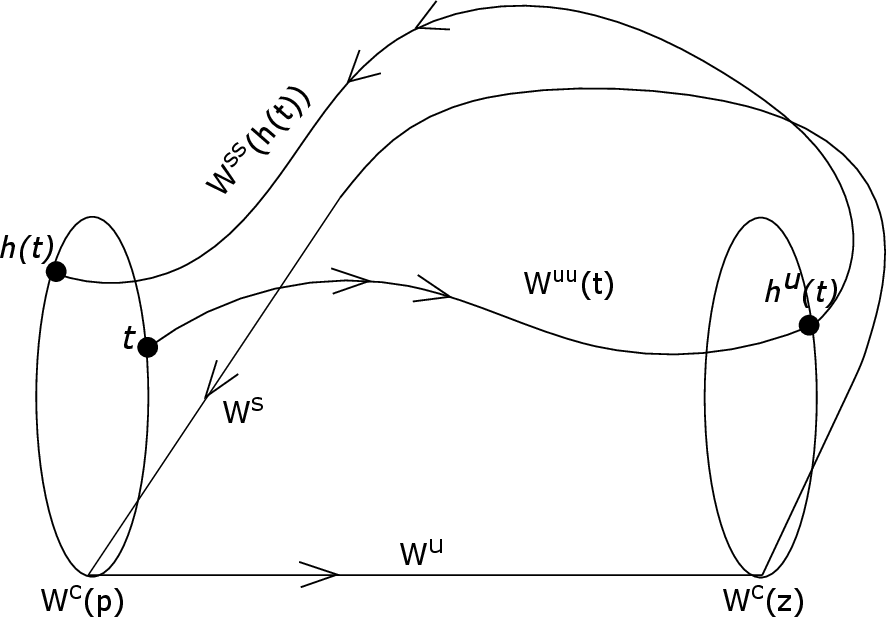}
    \caption{definition $h:\cW^c(p)\to \cW^c(p)$ for class $\nA$.}\label{figure.ccl}
\end{figure}

Fix the periodic point $\tp\in \tM$ and $\tz\in \tM$ a homoclinic 
point for $\tp$, i.e: $\tz\in W^s(\tp)\cap W^u(\tp)$, let us call $K=\cW^c(\tp)$. 
Define $h:K\to K$ (see figure~\ref{figure.ccl}) by 
$$h=h^s_{\tz,\tp}\circ h^u_{\tp,\tz},$$
and let $H^A_t:\real^{2d}\to \real^{2d}$
$$H^A_t=H^s_{\left(\tz,h^u_{\tp,\tz}(t)\right),\left(\tp,h(t)\right)} H^u_{\left(\tp,t\right),\left(\tz,h^u_{\tp,\tz}(t)\right)}.$$

Then if $m$ is an $su$-state, by proposition~\ref{p.sudisintA}, we have that ${H^A_t}_{\ast}m_t=m_{h(t)}$ for $\mu^c_{\tp}$-almost every $t\in K$.

\subsection{Class $\nB$}

Now we define $h$ and $H$ for $f$ of class $\nB$.

Observe that the restriction of $f$ to one of this compact leaves, $\cW^c(p)$, is a rotation in $\real/T\integer$ where $T$ is the period $p$ by the flow. So after a $C^1$ change of coordinates
we can suppose that $f_p=f\mid_{\cW^c(p)}=r_{\frac{1}{T}}$ where 
$$r_{\frac{1}{T}}:\real/\integer\to \real/\integer,\quad t\mapsto t+\frac{1}{T}.$$
The restriction of the cocycle to $\cW^c(p)$ is a cocycle over a rotation with Lebesgue $\mu^c_p$ as invariant measure.

Fix a flow periodic point $p\in \supp(\mu)$, lets take $z\in \cW^{cs}(p)\cap \cW^{cu}(p)$, by invariance of the center stable and center unstable manifolds $\phi_t(z)\in \cW^{cs}(p)\cap \cW^{cu}(p)$ for every $t\in \real$. Observe also that the orbit of $z$ is non recurrent.

\begin{figure}[ht]
    \centering
    \includegraphics[scale=0.6]{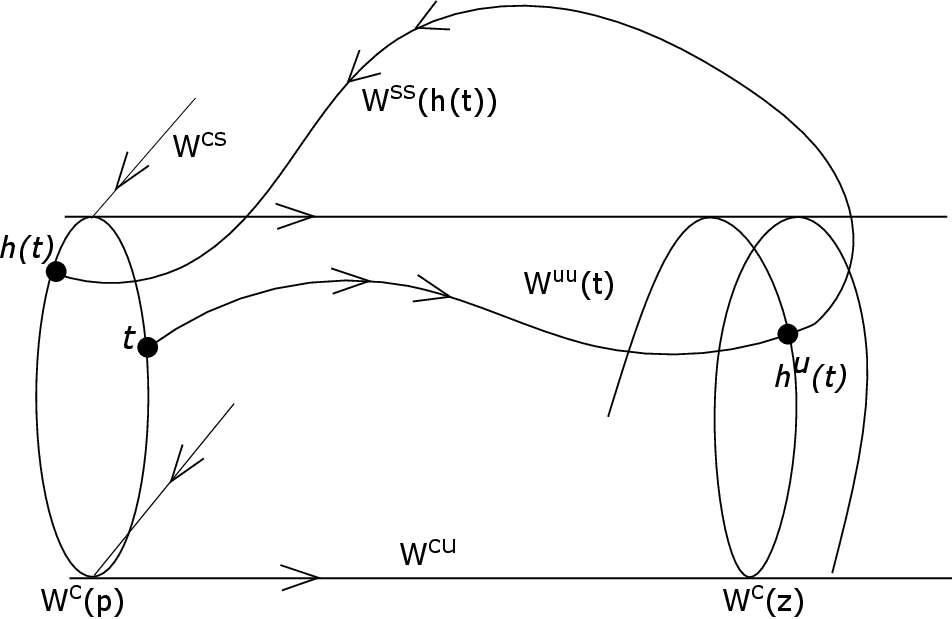}
    \caption{definition $h:\cW^c(p)\to \cW^c(p)$ for class $\nB$.}\label{figure.tom}
\end{figure}

Suppose that actually $z\in W^u(p)$ and define 
$h(p)=\cW^c(p)\cap W^s(z)$, now for $t\in \left[ 1,T \right]$ define 
$h(\phi_t(p)):=\phi_t(h(p))$ and observe that by the invariance of the stable and unstable manifolds $\phi_t(h(p))\in \cW^c(\phi_t(p))\cap W^s(\phi_t(z))$ and $\phi_t(z)\in W^u(\phi_t(p))$. 

So $h$ is a composition of stable and unstable holonomies. Also in the circle coordinates, identifying $p$ with $0$,
$$
h:\real/\integer\to \real/ \integer,\quad t\mapsto t+\omega,
$$ 
where $\omega$ is such that $\phi_\omega (p)=h(p)$. In particular $h$ also preserves the Lebesgue measure $\mu^c_p$ in the circle coordinates.

Call $h^u_{p,z}:\cW^c(p)-\{p\}\to \cW^c(z)$ the map given by $t\mapsto \phi_t(z)$, by the same reasoning as before this map is given by an unstable holonomy, if there is no risk of ambiguity we just write $h^u=h^u_{p,z}$. This map is not well defined in $0$ because $\phi_T(z)\neq z$.

Observe that as the center stable and center unstable are dense in $\supp(\mu)$ we can find points $z_1,\dots,z_d$ as before such that any pair $z_i$, $z_j$, with $i\neq j$, are in different orbits of the flow. Hence, we can define maps $h_1,\dots,h_d$ with the properties above.

\subsection{Weakly pinching}

Now we return to the proof of Theorem~\ref{teoprincipal}.
Fix once for all some periodic compact center leaf $K=\cW^c(p)$, take $A\mid_{K}:K\to \Sp$, this defines a linear cocycle over $f_{\tp}$ with invariant measure
 $\mu_K$, observe that $\mu_K$ is also $h$-invariant. Let $\lambda^1_{\mu_K}(t)\geq\cdots\geq \lambda^{k}_{\mu_K}(t)$ be the Lyapunov exponents of this cocycle.

 We say that $A\in H^{\alpha}(M)$ is \emph{Weakly pinching} if $L(A\mid_K,\mu_K)>0$.

\begin{lemma}\label{close}
 For every $A$ weakly pinching there exists $A'\in H^{\alpha}(M)$, arbitrary close to $A$ with that $A\mid_{K}=A'\mid_K$, such that $A'$ does not admit any $su$-invariant measure. 
\end{lemma}
\begin{proof}

Lets call the Lyapunov exponents of $(f_p,A'_p)$ relative to $\mu_K$ by $\lambda^1_c,\dots,\lambda^k_c$. By the positivity of the integrated Lyapunov exponents there 
 exists $i$ such that $\lambda^i_c(t)>\lambda^{i+1}_c(t)$ for some 
 $1\leq i\leq 2d$, as we have finite possibilities of $i$ we can take some 
 $\mu_K$ positive measure set $I\subset K$ such that, for every $t\in I$, there exists an invariant decomposition $\real^{2d}=V^+_t+V^-_t$ 
 with constant dimensions such that the smaller Lyapunov exponent in $V^+_t$ is larger than the largest Lyapunov exponent in $V^-_t$. Assume, without loose of generality, that $\dim(V^-)\geq \dim(V^+)$.
 
  We can also assume, reducing $I$ if necessary, that this decomposition varies continuously in $I$. As $h$ preserves the measure $\mu_K$ we can find an iterate $j$ such that $h^j(I)\cap I$ has positive measure, take $j$ to be the smaller integer with this property.
  
Suppose that $A'$ admits some $su$-invariant measure, by Theorem~\ref{t.sucinv} we can take the disintegration to be $su/c$-invariant.
By \cite[Proposition~3.1]{BP15} we have that, for every $t\in I$, $m_t=m^+_t+m^-_t$ where $\supp(m^+_t)\subset V^+_t$ and $\supp(m^-_t)\subset V^-_t$. 
Taking $t'\in I$ density point of $h^j(I)\cap I$ we can find a neighbourhood $N\subset K$ of $t'$ such that the sets $h^{u}_{p,z}(h^i(N))\subset \cW^c(z),i=0,\dots,j-1$ are disjoint. 

Let $H^{A'}_{t}$ be the composition of the unstable linear holonomy from $t$ to the point $h^u(t)$ and the stable linear 
holonomy from $h^u(t)$ to $h(t)$.

The $su/c$-invariance implies that $m_{h^j(t)}=(H^j_{t})_* m_t$ for $\mu^c_p$-almost every point. This implies that
$$
\supp(m_{h^j(t)})\subset ( V^+_{h^j(t)}\cup V^-_{h^j(t)})\cap (H^j_{t}V^+_t\cup H^j_{t}V^-_t).
$$

Fix a neighbourhood $\cN$ of $h^u(t')$ that does not contain $f^n(h^u(t'))$ for $n\neq 0$, making a small perturbation we find $\hA:M\to\Sp$ such that $\hA(x)=A'(x)$ outside $\cN$. We have that 
$${H^{\hA,j}_{t'}}=H^{A',j-1}_{h({t'})}\circ H^{s,A'}_{h^u({t'}),h({t'})}\circ \sigma \circ H^{u,A'}_{{t'},h^u({t'})}.
$$
Then lemma~\ref{lemma:symp_transv}  we can find $\sigma$ arbitrarily close to the identity such that 
$V^-_{h^j(t')}\cap H^j_{t'}V^+_{t'}=\{0\}$,  $V^-_{h^j(t')}\cap H^j_{t'}V^+_{t'}=\{0\}$ and $V^+_{h^j(t')}\cap H^j_{t'}V^+_{t'}=\{0\}$. So

$$
\supp(m_{h^j(t')})\subset V'_{h^j(t')}=V^-_{h^j(t)}\cap H^j_{t'}V^-_{t'}
$$
using lemma~\ref{lemma:symp_transv} again we can make $\dim V'<\dim V^-$. As $t'$ and $h^j(t')$ are density points of a set where the $V^+$ and $V^-$ varies continuously then there exists a positive measure subset of $t\in I$ such that $$
\supp(m_{h^j(t)})\subset V'_{h^j(t)}\quand \dim V'<\dim V^-.
$$ 

\begin{figure}[ht]
    \centering
    \includegraphics[scale=0.6]{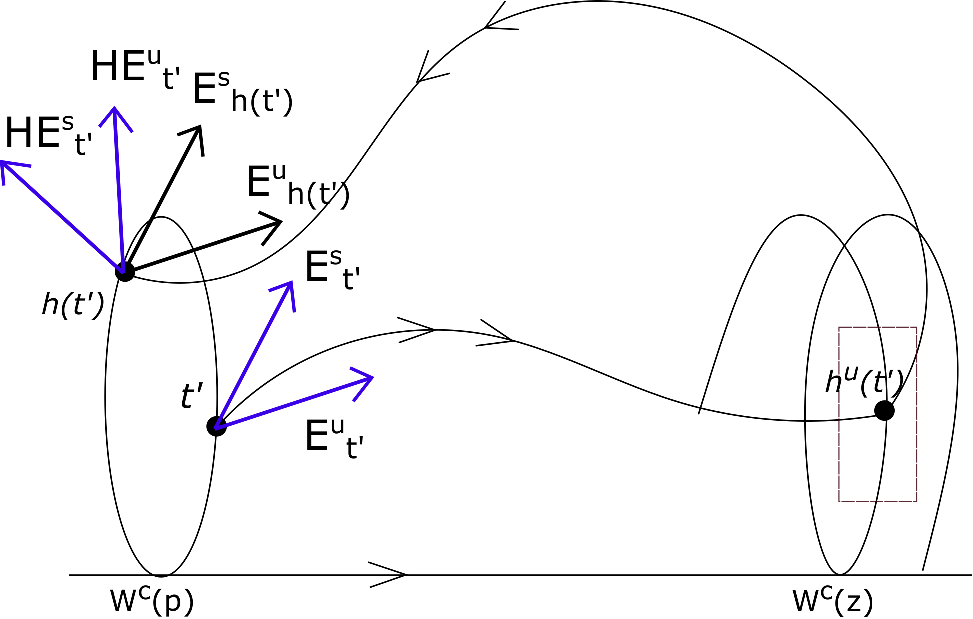}
    \caption{perturbation of $H$}\label{fig.perturb}
    \end{figure}

If $\hA$ still admits some $m$ $su$-invariant, we can repeat the argument taking a different homoclinic point $z_1,\dots,z_d$ (i.e the orbit of $z_i$ is not in the orbit of $z_j$ for $i\neq j$) with $V'$ instead of $V^-$ such that the new perturbation does not affect the cocycle in the orbit of $z$. Inductively, if the perturbed cocycles always admit some $su$-invariant measure, we can do this with
less than $d$-point $z_i$ to get that $\supp(m_t)=\emptyset$, for $t$ in a positive measure set, a contradiction. 

Then we can find $A''$ arbitrarily close to $A$ does not admit any $su$-invariant measure.

\end{proof}

Now we can prove Theorem \ref{teoprincipal}.
\begin{proof}[Proof of Theorem \ref{teoprincipal}]
 By Lemma \ref{close} there exists $A'$, arbitrary close to $A$, weakly pinching that does not admit any $su$-invariant measure, by the invariance principle (theorem~\ref{t.IP}) $A'$ has positive exponent and also by Proposition~\ref{prop.discontinuity} $A'$ is not accumulated by cocycles with zero exponents, then it is stably non-uniformly hyperbolic.
\end{proof}

\section{Proof of Theorem \ref{t.compact.leaves} and \ref{t.time.one}}\label{s.proofAB}
We need the following Proposition:

\begin{proposition}\label{p.densepinching}
Let $f:M\to M$ and $\mu$ be of class $\nA$ or $\nB$, then
there exist a dense set of $H^{\alpha}(M)$ weakly pinching.
\end{proposition}

\begin{proof}

To prove Proposition \ref{p.densepinching} let us recall some results.

We say that an $f$-invariant measure $\mu$ is \emph{non-periodic} if for every $k\in \integer$ $f^k\mid_{\supp \mu}$ 
is not the identity map.
By Xu~\cite{Xu15} in a very general topology (including the $H^{\alpha}(M)$ topology) there exists a dense set of cocycles with
$L(A,\mu)>0$.
So if there exist some compact periodic center leaf $\cW^c(p)$ such that such that $\mu^c_{p}$ is non-periodic we are done.

If we are not in the previous case we can do the following:
first we need periodic points of arbitrarily large period,
\begin{itemize}
\item for class $\nA$: We are assuming that for every $\tf$-periodic point $\tp\in\tM$, there exist $k(\tp)$ such that
$f^{k(\tp)}_{\tp}\mid_{\supp \mu^c_{\tp}}=\id$. As $\tf$ is an hyperbolic homeomorphism there exist periodic points of arbitrary large period, then as $k(\tp)$ is at least the period of $\tp$ we have arbitrary large $k(\tp)$, 
\item for class $\nB$: If for every periodic point $p\in \supp(\mu)$, the invariant measure is periodic this means that $\frac{1}{T_p}$ is rational. $T_p$ can be taken arbitrarily large, so we have rational rotations with arbitrarily large period.
\end{itemize}

Observe also that in the periodic case the Lyapunov exponents at a point $t\in \cW^c(p)$ are the logarithm of the eigenvalues of $A^{k}(t)$, where $k$ is the period, lets call the $L(A(t))$ the logarithm of the largest eigenvalue of $A^{k}(t)$.

Now, taking 
$$
A_{\theta}=\left(\begin{array}{cc}
\cos(\theta)\id_d & \sin(\theta)\id_d\\
-\sin(\theta)\id_d &\cos(\theta)\id_d \ \\
\end{array}\right)
A,
$$ 
for a generic $A$ we have that for an $n$-periodic point there exists an $O(\frac{1}{n})$ dense set of $\theta\in [0,2\pi]$ with $L(A_{\theta}(t))>0$ (see \cite[Lemma~3.4]{Xu15}), so taking $\cW^c(p)$ such that $k$ is very large we can take $\theta$ small such that $A_\theta$ is close to $A$.
So we have at least one point $t\in \cW^c(p)$, in the support of $\mu^{c}_{p}$, such that $A^{k}$ has positive eigenvalues, then as all the points in $\cW^c(p)$ have period $k$, by continuity of the eigenvalues there exists an open set, containing $t$, with positive eigenvalues, then as this set has positive $\mu^c_{p}$ measure we have that $L(A\mid_{\cW^c(p)},\mu^c_{p})>0$, as we wanted.

\end{proof}

Now we can prove Theorem~\ref{t.compact.leaves} and \ref{t.time.one}
\begin{proof}
 Take any $A\in H^{\alpha}(M)$, by Proposition~\ref{p.densepinching} arbitrarily close to $A$
 there exists $\hA\in H^{\alpha}(M)$ weakly pinching, now by Theorem~\ref{teoprincipal} there exists an open 
 set with positive Lyapunov exponents arbitrarily close to $A$, then $A$ is accumulated by open sets with positive Lyapunov exponents.
\end{proof}

\section{Proof of Theorem \ref{teo.continuity}}\label{s.proof.continuity}

In $\SL$ we have that, if $\lambda^+(t)>0$, $\real^2=E^+_t+E^-_t$. Lets convention that, if $\lambda^+(t)=0$, $\real^2=E^+_t=E^-_t$.
We say that $A$ is \emph{Weakly twisting} if there exists $\tilde{K}\subset K$ with $\mu_{K}(\tilde{K})>0$ and $j\in\natural$ such that
 $(H^A_t)^j\left(\lbrace E^+_t,E^-_t \rbrace\right) \cap \lbrace E^+_{h^j(t)},E^-_{h^j(t)} \rbrace=\emptyset$.
By the same arguments of Lemma~\ref{close} it is clear that weakly twisting implies that $A$ does not admit any $su$-invariant measure.

To prove Theorem~\ref{teo.continuity} we need the next lemma, whose proof is in \cite{BP15}.
\begin{lemma}\label{teo.LyapvsOsel}
 Let $H^{\alpha}(M)$ be the space of $\alpha$-H\"older linear cocycles over a dynamical system $g:M\to M$ and let
 $\mu$ be an ergodic $g$-invariant measure. If $A$ is a continuity point for 
  $L(\cdot,\mu):H^{\alpha}(M)\to \real$   with $L(A,\mu)>0$, then the Oseledets spaces are continuous in measure. 
  This means that if $A_k\to A$ then, for every $\epsilon>0$, 
  $\mu\left(\lbrace t\in K\textrm{, such that } \angle\left(E^{u,A_k}_t,E^{u,A}_t \right) > \epsilon \rbrace \right) \to 0 $,
  where $\angle\left(V,W\right)$ is the angle between $V$ and $W$.
\end{lemma}

We prove the same result for non ergodic measures:
\begin{lemma}\label{weakcont}
Assume that $A:K\to \SL$ is non-uniformly hyperbolic and a continuity point for Lyapunov exponents, 
then it is a continuity point, in measure, for the Oseledets decomposition.
\end{lemma}
\begin{proof}
Take an ergodic decomposition of $\mu_K$, $\{\mu_E,E\in \cE  \}$, where $\cE$ is the partition given by the ergodic decomposition.
Observe that if $t\in E$ then $\lambda^+_A(t)=\lambda^+_A(E)$.

Suppose by contradiction that there exist $\delta>0$ and $A_n\to A$ such that 
$\mu_K\{t\in K,\angle(E^{*,A_n}_t,E^{*,A}_t)>\delta   \}>\delta$ for $*=-$ or $+$. Here we use the convention that if 
$\dim(E)\neq \dim(F)$, $\angle(E,F)=\pi$.
Take a subsequence of $\lambda^+_{A_{n_k}}$ that converges for $\mu_K$-almost everywhere to $\lambda^+_{A}$.
This implies that $\lambda^+_{A_{n_k}}(E)\to \lambda^+_{A}(E)$ for almost every $E\in \cE$, 
by lemma~\ref{teo.LyapvsOsel} applied to every ergodic component we have that
$\mu_E\{t\in K,\angle(E^{*,A_{n_k}}_t,E^{*,A}_t)>\delta\}\to 0$, for almost every $E\in \cE$.
Then, by dominated convergence
$$
\mu_K\{t\in K,\angle(E^{*,A_{n_k}}_t,E^{*,A}_t)>\delta\}=\int \mu_E\{t\in K,\angle(E^{*,A_{n_k}}_t,E^{*,A}_t)>\delta\} \,d\mu_K
$$
converges to zero. 
This contradiction proves the Lemma.

\end{proof}

\begin{lemma}\label{stable}
 Let $A:M \to \SL$ be weakly twisting and weakly pinching, 
 and suppose that $A\mid_{K}:K\to \SL$ is a continuity point for the Lyapunov exponents. Then it is stable weakly twisting.
\end{lemma}
\begin{proof}
 Let $K'\subset K$ be the set of $t$ such that $\lambda(A\mid_{K},\mu_K,t)>0$. 
 Reducing $\tilde{K}\subset K'$ given by the weakly twisting definition if necessary, we can assume that there exists $\epsilon>0$ such that 
 $$
 \min_{a,b\in \lbrace +,- \rbrace} \angle \left((H^A_t)^j E^{a}_t,E^{b}_{h^j(t)} \right)>\epsilon
 $$ 
 for every $t\in \tilde{K}$. Take $c>0$ such that $\mu_K(\tilde{K})>2c>0$ and 
 take $0<\delta<\frac{\epsilon}{6}$ such that $\angle\left((H^A_t) V,(H^A_t) W \right)< \frac{\epsilon}{6}$ 
for every $V,W \in \proj$ with $\angle\left(V,W \right)< \delta$.

Now, by the continuity of the Oseledets spaces given by Lemma~\ref{weakcont}, 
for every $B\in H^{\alpha}(M)$ sufficiently close to $A$ there exists $\hat{K}\subset K'$ with $\mu_K(\hat{K})>\mu_K(K')-\frac{c}{3}$ 
such that $\angle\left(E^{\ast,B}_t,E^{\ast,A}_t \right)< \delta $, $\ast\in\lbrace +,- \rbrace$.

Moreover, since in $H^{\alpha}(M)$, $A\mapsto H^A_t$ varies continuously with respect to $A$, 
we have that for $B$ sufficiently close to $A$ 
$$\angle\left(H^B_t E^{\ast,B}_t,H^A_t E^{\ast,B}_t \right)< \frac{\epsilon}{6}.$$
So, taking $K''=\hat{K}\cap \tilde{K}$ we have that $\mu_K(K'')>\frac{2c}{3}$ and 
$\mu_K(h^j(K'')\cap \hat{K})>\frac{c}{3}$.
Therefore for every $t\in h^{-j}\left(h^j(K'')\cap \hat{K}\right)$ we have 
$$
\angle\left((H^B_t)^j E^{\ast,B}_t,(H^A_t)^j E^{\ast,A}_t \right)< \frac{\epsilon}{3}
\quand
\angle\left(E^{\ast,B}_{h^j(t)},E^{\ast,A}_{h^j(t)} \right)< \frac{\epsilon}{3}.
$$
Consequently, 
$$
\min_{a,b\in \lbrace +,- \rbrace} \angle\left((H^B_t)^j E^{a,B}_t,E^{b,B}_{h^j(t)} \right)>\frac{\epsilon}{3}.
$$
\end{proof}

Now we can prove Theorem \ref{teo.continuity}.
\begin{proof}[Proof of Theorem \ref{teo.continuity}]
 By Lemma \ref{close} there exists $\hA$, arbitrary close to $A$, 
 with the weakly twisting and weakly pinching property, such that $A\mid_{K}=\hA\mid_K$. 
 By Lemma \ref{stable} $\hA$ is stable weakly twisting then, by Proposition~\ref{prop.discontinuity}, $\hA$ is a continuity point for the Lyapunov exponents.
\end{proof}

\begin{proof}[Proof of Theorem \ref{teo.example.cont}]
Take the periodic point $p\in M$ (for simplicity assume that it is fixed) and consider $K=\{p\}\times S^1$. 
Let $r_{\theta(p)}:S^1\to S^1$ be the irrational rotation $t\mapsto t+\theta(p)$. 

By the unique ergodicity of the irrational rotation $\mu_K$ is the Lebesgue measure on $S^1$. So by
hypothesis $\big(r_{\theta(p)},B(p,\cdot)\big)$ is a $\epsilon$-monotonic quasi-periodic cocycle.
Also, by \cite{Av11}, we can find $\tB$ arbitrarily close to $B$ such that $L(\tB(p,\cdot),\mu_K)>0$. 
Using \cite[Theorem~3.8]{AvKr13} we have that $\big(R_{\theta(p)},\tB(p,\cdot)\big)$
 is a continuity point for the Lyapunov exponents. Then $(f,\tB)$ 
 falls into the hypotheses of Theorem~\ref{teo.continuity}. Hence for every $(f,B)\in H_p^{\epsilon}$
  there exists an open subset of $H^{\epsilon}_p$ arbitrarily close to  $B$  of continuity points for the Lyapunov exponents.    
\end{proof}

The proof of Theorem~\ref{t.continuity.B} is analogous.
\appendix
\section{closedness of $s$ and $u$-states}\label{appendix}
In the study of Lyapunov exponents of linear cocycles over hyperbolic or partially hyperbolic maps one of the principal tools to prove positivity, simplicity or continuity is to analyze the invariant measures of the cocycle that projects to some fixed invariant measure in the base.

With some conditions that allow the existence of linear stable and unstable holonomies, having zero exponents (in some cases also discontinuity) can be caracterized by some rigidity condition in the invariant measures of the cocycles, this is known as the \emph{Invariance Principle} (see \cite{Extremal}). This rigidity condition says that the measures must be $s$ and $u$-states, this means that the disintegration is invariant by the holonomies (see section~\ref{s.cocycles} for the precise definition).

In many works closedness of $s$ or $u$-states has been proved and used for specific cases (\cite{BBB}, \cite{Extremal}, \cite{ASV13}). The purpose of this appendix is to give a more general proof of this fact for general partially hyperbolic maps without any extra conditions on invariant measure of the base map. 

The precise statement of the main result is given in theorem~\ref{teo}.

\subsection{Smooth cocycles}\label{s.cocycles}
Let $\cE$ be a compact manifold, and let $F:M\times \cE\to M\times \cE$ be a smooth cocycle over $f$, this means that if $P:M\times \cE\to M$ is the natural projection to the first coordinate $P\circ F=f\circ P$ and $x\mapsto F_x$ is H\"older continuous to the topology of $C^r$ diffeomorphisms.

We say that $F$ admits \emph{stable holonomies} if for every $x,y\in M$, $x\rels y$, there exists $H^s_{x,y}:\cE\to\cE$ with the following properties: 
\begin{itemize}
\item $x\rels y\rels z$, $H^s_{x,z}=H^s_{y,z}\circ H^s_{x,y}$ and $H^s_{x,x}=\id$,
\item $F_y\circ H^s_{x,y}=H^s_{f(x),f(y)}\circ F_x$,
\item $(x,y,\xi)\to H^s_{x,y}(\xi)$ is continuous where $(x,y)$ varies in the set of points $x\rels y$,
\item there exist $C>0$ and $\gamma>0$ such that $ H^s_{x,y}$ is $(C,\gamma)$ H\"older for every $x\rels y$. 
\end{itemize}
Analogously we say that $F$ admits \emph{unstable holonomies} if for every $x\relu y$ there exist $ H^u_{x,y}$ with the same properties changing stable by unstable.
 
From now on fix $f$ and vary the cocycles $F$ projecting to $f$ in a topology such that $x,y,F\mapsto H^{s,F}_{x,y}$ varies continuously. 

Fix some $f$-invariant probability measure $\mu$, as $\cE$ is compact there always exists some $F$-invariant probability measure $m$ that projects to $\mu$. By Rokhlin disintegration theorem, we can disintegrate $m$ with respect to the partition given by the fibers $\{x\} \times \cE$, so we have $x\mapsto m_x$ defined almost everywhere.  

We say that an $F$-invariant measure that projects to $\mu$ is an \emph{$s$-state} if there exists a total measure subset $M'\subset M$ such that for every $x,y \in M'$, $x\rels y$, ${H^s_{x,y}}_*m_x=m_y$. Analogously, we say that a measure is an \emph{$u$-state} is the same is true changing stable by unstable manifolds. We call $m$ an $su$-state if it is booth $s$ and $u$-state.

We want to prove that
\begin{theorem}\label{teo}
If $m^k$ are $s$-states for $F_k$, that projects to $\mu$ such that $F_k\to F$ and $m^k\to m$  in the weak$^*$ topology then $m$ is an $s$-state. 
\end{theorem}

By \cite[theorem~4.1]{ASV13} if a cocycle $F$ has all his Lyapunov exponents equal to zero, then the $F$-invariant measure $m$ is an $su$-state. As a corollary we have
\begin{corollary}
If $F$ does not admit any $su$-state, then there exists a neighborhood of $F$ with non-zero exponents.
\end{corollary}

\subsection{Proof}

First we need to recall the Markov construction of \cite{ASV13}:
Given any point $x\in \supp(\mu)$ we can find some section $\Sigma$ transverse to the stable foliation, some $N>0$, $R>0$, $0<\delta<\frac{R}{2}$ and a measurable family $\lbrace S(z),z\in \Sigma \rbrace$ such that 
\begin{itemize}
\item $\cW^s(z,\delta)\subset S(z)\subset \cW^s(z,R)$ for all $z\in \Sigma$,
\item for all $l\geq 1$, $z,\zeta\in \Sigma$, if $f^{lN}(S(z))\cap S(\zeta)\neq \emptyset$ then $f^{lN}(S(z))\subset S(\zeta)$.
\end{itemize}
As taking an iterate will not affect our argument we suppose that $N=1$.

For each $z\in \Sigma$, let $r(z)$ be the largest integer such that $f^j(S(z))$ does not intersect any $S(w)$ for all $w\in \Sigma$, $0<j\leq r(z)$. Now let $\cB_0$ be the $\sigma$-algebra of sets $E\subset M$ such that for every $z$ and $j$ as before, either $E$ contains $f^j(S(z))$ or is disjoint from it. A $\cB_0$-mensurable function, is a function that is constant on the sets $f^j(S(z))$, $0\leq j\leq r(z)$.

For every $k\in\natural$, let $\cH_k:M\times\cE\to M\times \cE$ be defined by $(x,\xi)\to (x,{H_k}_x(\xi))$ where
\begin{equation}
{H_k}_x=\left\lbrace\begin{array}{cc}
H^{s,k}_{x,f^j(z)} & \text{ if }x\in f^j(S(z))\text{ for some }z\in \Sigma \\
\id & \text{otherwise}
\end{array}\right.
\end{equation}
where $H^{s,k}_{x,z}$ is the stable holonomy of $F_k$.

Now as in \cite{ASV13} we can change our cocycle by
$\tilde{F}_k=\cH_k F_k (\cH_k)^{-1}$, this is called a deformation cocycle of $F_k$, such that $x\mapsto \tilde{F_k}_x$ is $\cB_0$ measurable.

Let $m^k$ be an $F_k$-invariant measure, define $\tm^k=\cH^k_* m^k$, this measure is $\tilde{F}_k$-invariant.
Observe that $m^k$ being an $s$-state implies  that $x\mapsto m^k_x$ is $\cB_0$ measurable. Moreover, $m^k$ is an $s$-state if and only if this is true for every $z\in M$ and $\Sigma\ni z$ transversal to the stable foliation (this is explained in more detail in \cite[Section~4.4]{ASV13}).

\begin{lemma}\label{l.measurable1}
Let $\phi:M\times\cE\to \real$ be a measurable bounded function such that $v\mapsto \phi(x,v)$ continuous, then $\int \phi d m^k\to \int \phi d m$.
\end{lemma}
\begin{proof}
Fix $\varepsilon>0$ and take a compact set $K\subset M$ such that $\mu(K)>1-\frac{\varepsilon}{\norm{\phi}}$ and $\phi$ is continuous in $K\times \cE$, take   $\phi':M\times \cE \to \real$ be a continuous function such that $\phi(x,v)=\phi'(x,v)$ for every $x\in K$, $v\in \cE$ and $\norm{\phi'}\leq \norm{\phi}$. 

Now
take $k$ sufficiently large such that $\abs{\int \phi' d m^k- \int \phi' d m}<\epsilon$, then 
$$
\abs{\int \phi d m^k- \int \phi d m}\leq \abs{\int \phi' d m^k- \int \phi' d m}+2 \varepsilon.
$$

So for $k$ sufficiently large this is less than $3\varepsilon$, concluding the proof.
\end{proof}

\begin{lemma}
If $m^k\to m$ then $\tm^k\to \tm=\cH_* m$.
\end{lemma}
\begin{proof}
Let $\varphi:M\times \cE\to \real$ be a continuous function, then
$$
\int \varphi d\tm^k=\int \varphi\circ \cH_k d m^k,
$$
$\varphi\circ \cH$ is measurable but $v\mapsto \varphi\circ \cH(x,v)$ is continuous for every $x\in M$, then by lemma~\ref{l.measurable1} we have that 
\begin{equation}\label{eq.limite}
\int \varphi\circ \cH d m^k\to \int \varphi d \tm.
\end{equation}

Fix some $\epsilon>0$ and take $\delta>0$ such that $\d(a,b)<\delta$ implies that $\abs{\varphi(a)-\varphi(b)}<\epsilon$.
Now, the uniform convergence of the holonomies implies that for $k$ sufficiently large 
\begin{equation}\label{eq.uniform}
\d(\cH_k(x,v),\cH(x,v))<\delta.
\end{equation}

Observe that
$$
\abs{\int \varphi\circ \cH_k d m^k- \int \varphi d \tm}\leq 
\abs{\int \varphi\circ \cH_k d m^k- \int \varphi\circ\cH d m^k}+\abs{\int \varphi\circ \cH d m^k- \int \varphi d \tm}
$$
then by \eqref{eq.limite} and \eqref{eq.uniform} we have that for $k$ sufficiently large 
$$
\abs{\int \varphi\circ \cH_k d m^k- \int \varphi d \tm}<2\varepsilon.
$$
\end{proof}

So we are left to prove that $\tm^k\to \tm$ implies that $x\mapsto \tm_x$ is also $\cB_0$ measurable.

Let $\cB_0$ be a $\sigma$-algebra, let $\mu$ be a measure in $M$. Assume that we have some measures $\tm^k$ in $M\times \cE$ converging in the weak$^*$ topology to $\tm$ and let $P:M\times \cE \to M$ be the natural projection, also assume that $P_* \tm^k=\mu$. 

The next lemma is a corollary of lemma~\ref{l.measurable1}.
\begin{lemma}\label{l.measurable}
Let $\phi:M\to \real$ be a measurable function and $\varphi:\cE\to \real$ continuous, then $\int \phi\times \varphi d \tm^k\to \int \phi \times \varphi d\tm$.
\end{lemma}
Suppose that $x\mapsto \tm^k_{x}$ is $\cB_0$ measurable, this is true if and only if for every continuous function $\varphi:\cE\to \real$, $x\mapsto \int \varphi d\tm^k_{x}$ is $\cB_0$ measurable.

First we need the next lemma
\begin{lemma}\label{l.closed}
Let $\phi_k$ be a sequence of function in $L^2(\mu)$ that is $\cB_0$ measurable such that $\phi_k$ converges weakly to $\phi$, then $\phi$ is $\cB_0$ measurable.
\end{lemma}
\begin{proof}
First observe that the space of $\cB_0$ measurable functions is closed and convex in $L^2(\mu)$, lets call this space by $H\subset L^2(\mu)$. Suppose that $\phi\notin H$ then by Hahn-Banach there exist some $\rho\in L^2(\mu)$ such that $\int \rho \xi d\mu=0$ for every $\xi\in H$ and $\int \rho \phi d\mu>0$.
A contradiction because $\int \rho \phi_k d\mu\to\int \rho \phi d\mu$
\end{proof}
Now to conclude the proof of theorem~\ref{teo} we prove:
\begin{proposition}
$x\mapsto \int \varphi d\tm_{x}$ is $\cB_0$ measurable.
\end{proposition}
\begin{proof}
Take $\phi_k(x)=\int \varphi d\tm^k_{x}$, this function bounded then it is in $L^2(\mu)$. For any function $\rho\in L^2(\mu)$ we have that 
$$
\int \rho\phi_k d \mu=\int \rho(x)\int \varphi(v)d\tm^k_{x}(v)d\mu=\int \rho\times \varphi d\tm^k
$$
So by lemma~\ref{l.measurable} we have that 
$$
\int \rho\phi_k d \mu\to \int \rho\phi d \mu,
$$
where $\phi(x)=\int \varphi d\tm_{x}$. By hypothesis $\phi_k$ is $\cB_0$ measurable, then by lemma~\ref{l.closed} $\phi$ is $\cB_0$ measurable.
\end{proof}

%For acknowledgements section, please don't number the section, please begin it with \section*{Acknowledgements}
\section*{Acknowledgments} Thanks to M. Viana for the orientation, E. Pujals, C. Matheus and Disheng Xu for the discussions and useful ideas. L. Backes, F. Lenarduzzi, K. Marin and A. Sanchez for the useful commentaries. Mateus Sousa for the idea of the proof of lemma~\ref{l.closed}.

% You may incorporate your references as follows in your main tex file.
% Using BibTex is not recommended but can be handled.

\bibliography{bib}
\bibliographystyle{plain}

\end{document}